\documentclass[12pt]{amsart}
\usepackage{latexsym, amsmath, amssymb, amsthm, mathrsfs}
\usepackage{paralist}
\usepackage[colorlinks=true,linkcolor=blue,urlcolor=blue, citecolor=blue]%
  {hyperref}
\usepackage[T1]{fontenc}
\usepackage{mathptmx}
\usepackage{microtype}

\usepackage{comment}
\usepackage{MnSymbol}
\usepackage[]{graphicx}
\usepackage{tikz}
\usepackage[centering, includeheadfoot, hmargin=1in, vmargin=1in,
  headheight=30.4pt]{geometry}
\newtheorem{lemma}{Lemma}[section]
\newtheorem{theorem}[lemma]{Theorem}
\newtheorem{thm}[lemma]{Theorem}
\newtheorem{corollary}[lemma]{Corollary}
\newtheorem{cor}[lemma]{Corollary}
\newtheorem{proposition}[lemma]{Proposition}
\newtheorem{exm}[lemma]{Example}
\newtheorem*{thm*}{Theorem}
\newtheorem*{def*}{Definition}
\theoremstyle{definition}
\newtheorem{Def}[lemma]{Definition}
\newtheorem{remark}[lemma]{Remark}
\newtheorem{rem}[lemma]{Remark}

\renewcommand{\theequation}%
{\arabic{section}.\arabic{lemma}.\arabic{equation}}

\renewcommand{\AA}{\ensuremath{\mathbb{A}}} 
\newcommand{\CC}{\ensuremath{\mathbb{C}}} 
\newcommand{\NN}{\ensuremath{\mathbb{N}}}
\newcommand{\PP}{\ensuremath{\mathbb{P}}} 
\newcommand{\RR}{\ensuremath{\mathbb{R}}} 
 
\newcommand{\ZZ}{\ensuremath{\mathbb{Z}}} 
\newcommand{\VV}{\ensuremath{\mathbb{V}}} 

\newcommand{\Osh}{\ensuremath{\mathcal{O}}} 
 
\newcommand\cF{\mathcal{F}}

\renewcommand{\geq}{\geqslant}
\renewcommand{\leq}{\leqslant}
\newcommand\wh{\widehat}
\newcommand\scp[1]{\langle #1 \rangle}
\DeclareMathOperator{\ev}{ev}

\DeclareMathOperator{\codim}{codim}

\DeclareMathOperator{\conv}{conv}

\DeclareMathOperator{\Ker}{ker}

\DeclareMathOperator{\Proj}{Proj}

\DeclareMathOperator{\reg}{reg}

\DeclareMathOperator{\Span}{span}
\DeclareMathOperator{\Spec}{Spec}

\DeclareMathOperator{\hf}{HF}

\newcommand{\mm}{\mathfrak{m}}

\newcommand\psd{positive semidefinite}

\begin{document}
\title{Do Sums of Squares Dream of Free Resolutions?}
\author{Grigoriy Blekherman}
\address{Grigoriy Blekherman, School of Mathematics, Georgia Institute of Technology, 686 Cherry Street, Atlanta GA, 30332, USA}
\email{greg@math.gatech.edu }
\author{Rainer Sinn}
\address{Rainer Sinn, School of Mathematics, Georgia Institute of Technology, 686 Cherry Street, Atlanta GA, 30332, USA}
\email{sinn@math.gatech.edu }
\author{Mauricio Velasco}
\address{Mauricio Velasco, Departamento de
  Matem\'aticas\\ Universidad de los Andes\\ Carrera 1 No. 18a 10\\ Edificio
  H\\ Primer Piso\\ 111711 Bogot\'a\\ Colombia \textrm{and} Departamento de
  Matem\'aticas y Aplicaciones\\ Universidad de la Rep\'ublica (CURE)\\
  Maldonado\\ Uruguay}
\email{mvelasco@uniandes.edu.co or mvelasco@cmat.edu.uy}


\begin{abstract}
For a real projective variety $X$, the cone $\Sigma_X$ of sums of squares of linear forms plays a fundamental role in real algebraic geometry.
The dual cone $\Sigma_X^*$ is a spectrahedron and we show that its convexity properties are closely related to homological properties of $X$. For instance, we show that all extreme rays of $\Sigma_X^*$ have rank one if and only if $X$ has Castelnuovo-Mumford regularity two. More generally, if $\Sigma_X^*$ has an extreme ray of rank $p>1$, then $X$ does not satisfy the property $N_{2,p}$. We show that the converse also holds in a wide variety of situations: the smallest $p$ for which property $N_{2,p}$ does not hold is equal to the smallest rank of an extreme ray of $\Sigma_X^*$ greater than one. We generalize the work of Blekherman-Smith-Velasco on equality of nonnegative polynomials and sums of squares from irreducible varieties to reduced schemes and classify all spectrahedral cones with only rank one extreme rays. Our results have applications to the positive semidefinite matrix completion problem and to the truncated moment problem  on projective varieties.
\end{abstract}

\maketitle

\section{Introduction}
Minimal free resolutions and spectrahedra are central objects of study in commutative algebra and convex geometry respectively. We connect these disparate areas via real algebraic geometry and show a surprisingly strong connection between convexity properties of certain spectrahedra and the minimal free resolution of the defining ideal of the associated real variety. In the process, we address fundamental questions on the relationship between nonnegative polynomials and sums of squares.

In real algebraic geometry, we associate two convex cones to a real projective variety $X$: the cone $P_X$ of quadratic forms that are nonnegative on $X$, and the cone $\Sigma_X$ of sums of squares of linear forms. 
A recent line of work shows that convexity properties of these cones are strongly related to geometric properties of the variety $X$ over the complex numbers \cite{BleNP},\cite{BSV}, and \cite{BIKV}. We extend these novel connections into the realm of homological algebra by showing a direct link between the convex geometry of the dual convex cone $\Sigma_X^*$ and property $N_{2,p}$ of the defining ideal of $X$: for an integer $p\geq 1$, the scheme $X$ satisfies property $N_{2,p}$ if the $j$-th syzygy module of the homogeneous ideal of $X$ is generated in degree $\leq j+2$ for all $j<p$. 

The dual cone $\Sigma_X^*$ to sums of squares is naturally a spectrahedron, i.e. it is a section of the cone of positive semidefinite matrices with a linear subspace. By analogy with univariate polynomials, we call $\Sigma_X^*$ \textit{the Hankel spectrahedron} of $X$. The rank of an extreme ray of $\Sigma_X^*$ is, by definition, the rank of any symmetric matrix corresponding to a point spanning the ray. Our first main result is the following: if the Hankel spectrahedron $\Sigma_X^*$ has an extreme ray of rank $p>1$, then $X$ does not satisfy property $N_{2,p}$. 

The largest $p$ for which the property $N_{2,p}$ holds is called the \textit{Green-Lazarsfeld index} of $X$, see \cite{Brunsetal}. We call the smallest integer $p$ such that the Hankel spectrahedron has an extreme ray of rank $p>1$, \textit{the Hankel index} of $X$. Using this terminology, our first main result says that the Hankel index is always at least one more than the Green-Lazarsfeld index. Our second main result is that this inequality is in fact an equality in a wide range of situations: for varieties (reduced schemes) of regularity $2$, canonical models of general curves of genus at least $4$, and subspace arrangements defined by quadratic square-free monomial ideals. The proofs of our two main results appear in Section~\ref{SEC: mainThm1} and Section~\ref{SEC:BPF} respectively. The rest of the article is devoted to their applications to questions from real algebraic geometry, convex geometry, statistics, and real analysis.

\subsection{Applications} 
A projective variety $X$ has regularity two if the ideal of $X$ is generated by quadrics and all matrix entries in its minimal free resolution are linear forms. Equivalently, $X$ satisfies property $N_{2,p}$ for all natural numbers $p\geq 1$ and therefore its Green-Lazarsfeld index is infinite. It follows from our first main theorem that the Hankel index of $X$ is infinite, i.e. all extreme rays of the Hankel spectrahedron of $X$ have rank one and in particular $\Sigma_X^*=P_X^*$.
Conversely we classify all reduced schemes $X$ for which the cones $P_X$ and $\Sigma_X$ coincide, i.e. for which every nonnegative quadratic form on $X$ is a sum of squares. We show in Section~\ref{Sec:TRreg2} that this condition is in fact equivalent to $2$-regularity of $X$ whenever $X$ is totally real. This is a generalization of the work of \cite{BSV} from the case of irreducible varieties to reduced schemes. We find it remarkable that the generalization of varieties of minimal degree to reduced schemes in algebraic geometry, which are those with the smallest Castelnuovo-Mumford regularity, turns out to be the natural concept from the point of view of convex algebraic geometry as well.

In Section~\ref{ssec:rank1}, we use our results to solve a basic problem in convex geometry: classify spectrahedral cones all of whose extreme rays have rank one. 
We show that a spectrahedral cone $C$ has only rank one extreme rays if and only if $C$ is the Hankel spectrahedron of a reduced scheme $X$ of regularity two. Two-regular reduced schemes are completely classified in \cite{EGHP}, they consist of varieties of minimal degree which are \textit{linearly joined}; see Section \ref{Sec:TRreg2} for details. We thus obtain an explicit description of all spectrahedral cones with only rank one extreme rays. Examples of such cones were studied in \cite{AHMRMR960140, Hildebrand}, but a full classification takes us inevitably into the realm of algebraic geometry.

In Section~\ref{MatrixCompl}, we recover the positive semidefinite matrix completion theorem of Grone, et al. \cite{PSDComp,AHMRMR960140} which characterizes those partial matrices for which the positive semidefinite matrix completion problem is combinatorially as simple as possible. We show its equivalence to the theorem of Fr\"{o}berg on $2$-regularity of monomial ideals \cite{FroMR1171260} by recasting it as one about sums of squares on certain subspace arrangements. Fr\"{o}berg's theorem was generalized in \cite{EGHPMR2188445} where it is shown that the Green-Lazarsfeld index of a reduced scheme $X$ defined by quadratic monomials is equal to the length of the smallest chordless cycle of its associated graph $G$ minus $3$. We show the corresponding result holds for the Hankel index of $X$, obtaining a new extension of the positive semidefinite matrix completion theorem. It is surprising that, despite a significant amount of work in examining the extreme rays of such spectrahedra \cite{AHMRMR960140}, \cite{HeltonExtremePSD}, \cite{LaurentSeriesParallel}, \cite{LaurentSparsityOrder}, \cite{SolusUhlerYoshida}, this result was not observed. 

In statistics, we prove Theorem~\ref{thm:gaussian}, which gives necessary and sufficient conditions for the existence of a positive definite matrix completion. The existence of a positive definite completion is equivalent to the existence of the maximum likelihood estimator in a Gaussian graphical model.

In algebraic geometry, we find a new characterization of $2$-regular reduced schemes. Previously, \cite{EGHP} characterized $2$-regular schemes via a geometric property of \textit{smallness}; see Sections \ref{subsec:introbpf} and \ref{subsec:bpf} for details. 
We characterize $2$-regularity of $X$ in terms of properties of base-point-free linear series on $X$. Let $X\subseteq \PP^n$ be a reduced projective scheme with graded coordinate ring $R$. Then $X$ is $2$-regular if and only if every base-point-free linear series $W\subseteq R_1$ generates $R_2$. As explained in Section~\ref{subsec:HIndex} this characterization was motivated by our investigation of the Hankel index.

Finally, we also discuss applications to truncated moment problems of real analysis in Section \ref{SEC:Moments}. In its simplest form, a truncated moment problem asks whether a given linear operator $\ell$ on functions on $X$ is the integral with respect to some Borel measure. Our main result is a new sufficient criterion based on the Green-Lazarsfeld index of $X$, which is another motivation for its study. This is also of direct relevance in polynomial optimization, as it gives a new guarantee of exactness for semidefinite relaxations.

\subsection{Rank obstructions from base-point-free linear series}\label{subsec:introbpf} Motivated by the equivalent characterizations of schemes of regularity $2$ established in \cite{EGHP}, there has been some amount of work on the connection between property $N_{2,p}$ and a geometric property of the variety $X$, which is called $p$-smallness; $X$ is $p$-small if for every subspace $L$ of dimension at most $p$ for which $\Gamma:=L\cap X$ is finite, $\Gamma$ is linearly independent. If $X$ satisfies property $N_{2,p}$, then $X$ is $p$-small and sometimes the converse is known to hold \cite{GLCurves,EGHPMR2188445}. The key to the results in this article is the introduction of a property of base-point-free linear series on $X$, which we show to lie logically between $N_{2,p}$ and $p$-smallness. We say that $X$ has the $p$-base-point-free property if every base-point-free linear series of codimension at most $p$ generates the degree two part of the homogeneous coordinate ring of $X$. As we show in Theorem~\ref{Thm:JumpNumberpBp}, the $p$-base-point-free property on $X$ provides an obstruction for the existence of extreme rays in $\Sigma_X^*$ of rank $p$.
We think that the study of the $p$-base-point-free property is, in itself, an interesting problem of complex algebraic geometry. Specifically we would like to have methods for deciding whether a variety $X$ satisfies the $p$-base-point-free property or a characterization of those varieties for which this property coincides with either $N_{2,p}$ or $p$-smallness; see related examples in Section~\ref{SEC:BPF}.

\medskip
\textbf{Acknowledgments.}
We thank Gregory G. Smith for many useful conversations. We also thank the anonymous referees for their comments that helped to improve the paper.
The first two authors were partially supported by NSF grant DMS--0757212; the third author was partially supported by CSIC-Udelar and by the FAPA funds from Universidad de los Andes.

\section{Ranks of extreme rays and property $N_{2,p}$}\label{SEC: mainThm1}
We mostly work with real projective schemes $X\subset \PP^n$. We denote by $\AA^n$ the $n$-dimensional affine space over $\RR$, which, as a set, is $\CC^n$ and is equipped with the real Zariski topology. Its closed sets are therefore the zero-sets of real polynomials. The ring of regular functions on $\AA^n$ is the polynomial ring $\RR[x_1,\dots,x_n]$ in $n$ variables with coefficients in $\RR$. We mostly work with projective space over $\RR$, which we denote by $\PP^n$. As a set, it is the set of all lines in $\CC^{n+1}$, where we represent the line $\{ \lambda v\colon \lambda\in \CC\}$ by $[v]$ (for $v\in \CC^{n+1}\setminus\{0\}$). Its topology is the real Zariski topology whose closed sets are the zero-sets of homogeneous polynomials with real coefficients. The set of real points in $\PP^n$, denoted by $\PP^n(\RR)$, are all lines $[v]$ that can be represented by a vector $v\in\RR^{n+1}$.
More precisely, affine space $\AA^n$ is $\Spec(\RR[x_1,\dots,x_n])$ and $\PP^n = \Proj(\RR[x_0,\dots,x_n])$ (see \cite[Exercise II.4.7]{HarMR0463157} for how these definitions relate with those of the above discussion).

A reduced scheme $X\subset\PP^n$ is a closed subset of $\PP^n$ that is defined by a radical ideal $I\subset\RR[x_0,\dots,x_n]$. The homogeneous coordinate ring $\RR[X]$ of $X$ is the quotient ring $\RR[x_0,\dots,x_n]/I$. The fact that $X$ is reduced means that $\RR[X]$ has no nilpotent elements; so $\RR[X] = \RR[x_0,\dots,x_n]/I$, where $I$ is the vanishing ideal of $X$. Reduced projective schemes $X\subset\PP^n$ are in one-to-one correspondence with homogeneous radical ideals $I\subseteq \RR[x_0,\dots,x_n]$. 

A reduced scheme $X\subset\PP^n$ is totally real if the real points $X(\RR)$ are Zariski dense in $X$. Equivalently, the homogeneous ideal of polynomials vanishing on $X$ is real radical \cite[Section 4.1]{BochnakMR1659509}.

Throughout the paper, we write $S$ for the polynomial ring $\RR[x_0,\dots,x_n]$ and $S_d$ for the vector space of homogeneous polynomials of degree $d$. Let $X\subseteq \PP^n$ be a reduced totally real non-degenerate closed projective subscheme with defining ideal $I\subseteq S$ and homogeneous coordinate ring $R=S/I$. 

We write $P_X\subset R_2$ for the convex cone of quadratic forms which are nonnegative on $X(\RR)$ and $\Sigma_X\subset R_2$ for the convex cone of sums of squares of linear forms.
The dual cones $P_X^*$ and $\Sigma_X^*$ are defined as follows:
$$\Sigma_X^*=\{\ell \in R_2^* \mid \ell(p) \geq 0 \,\,\,\, \text{for all} \,\,\,\, p \in \Sigma_X\}.$$
and $$P_X^*=\{\ell \in R_2^* \mid \ell(p) \geq 0 \,\,\,\, \text{for all} \,\,\,\, p \in P_X\}.$$
The inclusion map $R_2^*\rightarrow S_2^*$  dual to the quotient map canonically embeds $\Sigma_X^*$ and $P_X^*$ into $S_2^*$. More specifically, the inclusion map identifies a linear functional $\ell \in R_2^*$ with the quadratic form $Q_\ell \in S_2^*$, $Q_\ell(p)=\ell(p^2)$, for all $p \in S_1$. Therefore, we think of the dual cones $\Sigma_X^*$ and $P_X^*$ as subsets of $S^*_2$. We write $S_+$ for the cone of positive semidefinite quadratic forms on $S_1$. We observe that under this identification the cone $\Sigma_X^*$ is a spectrahedron:
$$\Sigma_X^*=S_+ \cap I_2^\perp,$$
where $I_2^\perp\subset S_2^*$ is the subspace consisting of linear functionals $\ell\in S_2^*$ such that $\ell(p)=0$ for all $p \in I_2$. 

\begin{Def}
We call the dual cone $\Sigma_X^*$ the \emph{Hankel spectrahedron} of $X$. We call the smallest integer $p>1$ such that the Hankel spectrahedron has an extreme ray of rank $p$ the \emph{Hankel index} of $X$ and denote it by $\eta(X)$. Our convention is that the Hankel index is infinite if all extreme rays of the Hankel spectrahedron $\Sigma_X^*$ have rank $1$.
\end{Def}

Evaluation at a point $x\in \RR^{n+1}$ determines a map $\ev_x\colon S_2\rightarrow \RR$ called the real point evaluation at $x$. If $[x]\in \PP^n(\RR)$, then different affine representatives $\tilde{x}$ of $[x]$ in $\RR^{n+1}$ define point evaluations which are related by multiplication by a positive real number. Abusing terminology, we refer to the point evaluation at $[x]$ to mean the point evaluation at any such $\tilde{x}$. A form $\ell\in S_2^\ast$ has rank one if and only if $\ell$ or $-\ell$ is a real point evaluation. 

If $X$ is set-theoretically defined by quadrics, it follows that the extreme rays of $\Sigma_X^*$ of rank $1$ are the evaluations at points of $X(\RR)$. Moreover, the point evaluations at points of $X(\RR)$ are precisely the extreme rays of the dual cone $P_X^*$, cf.~\cite[Lemma 4.18]{BPTMR3075433}.
Hence, the Hankel spectrahedron $\Sigma_X^*$ must have extreme rays of rank strictly greater than one if the cones $\Sigma_X$ and $P_X$ are not equal. 

\subsection{Which spectrahedral cones are Hankel?}
Let $L$ be a linear subspace of $S_2^*$. 

\begin{lemma}\label{Lem:spectrahedral}
The spectrahedral cone $L\cap S_+$ is the Hankel spectrahedron of a reduced and totally real scheme that is set-theoretically defined by quadrics if and only if the linear space $L$ is spanned by elements of rank $1$.
\end{lemma}

\begin{proof}
If $X$ is reduced, totally real, and set-theoretically defined by quadrics, then a quadric $q\in S_2$ vanishes on $X$ if and only if it is annihilated by the evaluations at points of $X(\RR)$. Thus, the linear space $L =I_2^{\perp}$ of linear functionals vanishing on $I_2$ is spanned by elements of rank $1$.  

Conversely, suppose that $L$ is spanned by elements of rank $1$. Up to a sign, the rank $1$ elements of $L$ are precisely the rank $1$ extreme rays of $L\cap S_+$, which are point evaluations. Let $Z$ be the set of points in $\PP^n$ corresponding to these extreme rays. Then $L^{\perp}$ consists of the quadrics vanishing at all points in $Z$. Letting $X$ be the reduced subscheme of $\PP^n$ defined by the Zariski closure of $Z$, we have $\Sigma_X^*=L\cap S_+$, proving the claim.
\end{proof}

\subsection{Base-point-free linear series and the Hankel index.}\label{subsec:HIndex}
A linear series is a vector subspace $W\subseteq R_1$. We say that $W$ is base-point-free on $X$ if the linear forms in $W$ have no common zeros in $X$. In this section, we introduce the $p$-base-point-free property of linear series. We prove that this property is useful in bounding the Hankel index of $X$ and that it is closely related with the geometry of syzygies of $X$. Our methods also allow us to compute the Hankel index of several classes of varieties in Theorem~\ref{Thm:MainJN}.

\begin{Def}
Let $p>0$ be an integer. We say that $X\subseteq \PP^n$ satisfies the $p$-base-point-free property if the ideal generated in the homogeneous coordinate ring $R$ of $X$ by every base-point-free linear series $W$ of codimension at most $p$ contains $R_2$.  
\end{Def}

\begin{theorem}\label{Thm:JumpNumberpBp}
Let $X\subseteq \PP^n$ be a real non-degenerate closed subscheme. If $X$ satisfies the $p$-base-point-free property, then the Hankel index $\eta(X)$ is at least $p+1$.
\end{theorem}

\begin{proof}
Suppose $\RR_+\ell \subset S_2^*$ is an extreme ray of $\Sigma_X^*$ of rank greater than $1$. Write $B_\ell$ for the bilinear form on $R_1\times R_1$, which takes $(f,g)$ to $B_\ell(f,g) = \ell(fg)$. We will show that its kernel $W:=\ker(B_{\ell})\subseteq R_1$ is a base-point-free linear series on $X$.   
Assume for contradiction that there is a zero $\alpha\in \PP^n$ of $W$ in $X$. If $\alpha$ is real (resp.~complex), then the evaluation at $\alpha$ (resp.~the imaginary part of the evaluation at $\alpha$) defines a real linear functional $\ell_\alpha\in S_2^*$ with $\ell_\alpha(I_2)=0$ and with the property that ${\rm ker}(B_{\ell_\alpha})$ contains $W$.
Since $\ell$ is an extreme ray of the spectrahedron $\Sigma_X^*$, we conclude from \cite[Corollary 3]{RamanaGoldman} that $\ell=\lambda \ell_\alpha$ for some $\lambda>0$. This is impossible because $\ell_\alpha$ has rank one (resp.~because $\ell_\alpha$ is not positive semidefinite) and thus we conclude that $W$ is base-point-free on $X$. 
If $\ell$ has rank $\eta(X)\leq p$, then $I+(W)_2\supseteq (x_0,\dots,x_n)^2$ because $X$ satisfies the $p$-base-point-free property. This contradicts the existence of the linear functional $\ell$ which annihilates $(W)_2$. We conclude that the Hankel index of $X$ is at least $p+1$ as claimed. 
\end{proof}

\subsection{Base-point-free linear series and property $N_{2,p}$.}\label{subsec:bpf}
Let $k$ be a field of characteristic $0$ and let $S=k[x_0,\dots, x_n]$ be the polynomial ring with the usual grading by total degree. A finitely generated graded $S$-module $M$ has a unique minimal free resolution  
\[ \cdots \xrightarrow{\phi_{t+1}} F_t\xrightarrow{\phi_t} F_{t-1}\xrightarrow{\phi_{t-1}} \cdots \xrightarrow{\phi_1}F_0\rightarrow M\rightarrow 0\]
with $F_i=\bigoplus_{n\in \ZZ} S(-n)^{b_{i,n}}$ for some natural numbers $b_{i,n}$. We denote by $t_i(M)$ the largest degree of a minimal generator in the module $F_i$ with the convention that $t_i(M)=-\infty$ if $F_i=0$.
The module $M$ is $m$-regular if and only if $t_i(M)\leq m+i$ for all $i\geq 0$. The regularity of $M$ is the smallest such $m$.
In terms of local cohomology supported at the homogeneous maximal ideal, we can define the regularity of an artinian graded $S$-module as its largest nonzero degree and for any finitely generated module $M$ by the formula $\reg(M)=\max_{i\geq 0}\{i+\reg H^i_{\mm}(M)\}$~\cite[Section 4B]{EisMR2103875}.

Assume $X\subseteq \PP^n$ is a non-degenerate, i.e.~not contained in a hyperplane, closed subscheme and let $I\subseteq S$ be its saturated homogeneous ideal. The scheme $X$ is $m$-regular if the graded $S$-module $I$ is. The scheme $X$ satisfies property $N_{2,p}$ for an integer $p\geq 1$
if the inequality $t_i(I)\leq 2+i$ holds for all $0\leq i\leq p-1$. In other words, $X$ satisfies property $N_{2,p}$ if and only if $I$ is generated by quadrics and the first $p-1$ maps in its minimal free resolution are represented by matrices of linear forms.  

The homological property $N_{2,p}$ relates to the following geometric property, see \cite{EGHP} and \cite{GLCurves}.
\begin{Def} Let $p>0$ be an integer. We say that $X\subseteq \PP^n$ is \emph{$p$-small} if for every linear space $L\subseteq \PP^n$ of dimension at most $p$ for which $Y:=X\cap L$ is finite, the scheme $Y$ is linearly independent, i.e.~the dimension of its span is $1+\deg(Y)$.
\end{Def}

The following theorem shows that the $p$-base-point-free property lies in between property $N_{2,p}$ and $p$-smallness.

\begin{theorem}\label{thm:pBp}
Let $p\geq 1$ be an integer and let $X\subset\PP^n$ be a non-degenerate closed reduced scheme. Consider the properties
\begin{enumerate}
\item $X$ satisfies property $N_{2,p}$;
\item $X$ satisfies the $p$-base-point-free property;
\item $X$ is $p$-small.
\end{enumerate}
The implications $(1)\implies (2)\implies (3)$ hold. Moreover, any one of these properties holds for all $p\geq 1$ if and only if $\reg(X)=2$.
\end{theorem}
\begin{proof}
Corollary 5.2 in Eisenbud, Huneke, and Ulrich~\cite{EHU} establishes that $(1)\implies (2)$.
For $(2)\implies (3)$ assume $(2)$ and let $L$ be a linear subspace of dimension $\leq p$ for which $Y:=L\cap X$ is finite. Let $W$ be the vector space of forms defining $L$, which has dimension $n+1-p$, and let $\ell$ be a general linear form. Since $Y=L\cap X$ is finite, the linear series $W':=W+\langle \ell\rangle$ is base-point-free on $X$ and therefore $(W')$ contains $\mm^2$. In other words, the Hilbert function of the coordinate ring $R/(W')$ is $0$ beginning in degree $2$. We denote the Hilbert function of a graded ring $A$ by $\hf(A,t)$. Since $Y$ is finite, it is an arithmetically Cohen-Macaulay scheme and thus we obtain the equality  $\hf(R/(W'),t)=\hf(R/W,t) - \hf(R/W,t-1)$ of Hilbert functions. Inductively, starting with $t=2$, we conclude that $\hf(R/W,t) = \hf(R/W,1)$, so the degree of $Y$ equals $\hf(R/(W),1)$. This is precisely one more than the dimension of the projective span of $Y$, proving $(3)$.  
If $\reg(X)=2$, then by definition of regularity, the scheme $X$ satisfies property $N_{2,p}$ for all $p\geq 1$. If $X$ satisfies property $(3)$ for all $p\geq 0$, then $X$ is called a small scheme and $\reg(X)=2$ by the Main Theorem in Eisenbud-Green-Hulek-Popescu~\cite{EGHP}.
\end{proof}

\begin{cor}\label{Thm:JumpNumber}
Let $X\subseteq \PP^n$ be a non-degenerate closed reduced scheme. If $X$ satisfies property $N_{2,p}$, then the Hankel index $\eta(X)$ is at least $p+1$. 
\end{cor}
\begin{proof} Combine Theorem \ref{Thm:JumpNumberpBp} and Theorem~\ref{thm:pBp}.
\end{proof}

\begin{corollary}\label{cor:reg}
Let $X\subseteq \PP^n$ be a non-degenerate reduced scheme of regularity $2$. All extreme rays of the Hankel spectrahedron $\Sigma_X^\ast$ of $X$ have rank $1$, in particular $\Sigma_X^*=P_X^*$.
\end{corollary}
\begin{proof} Since $X$ is non-degenerate the condition on the regularity implies that $X$ satisfies $N_{2,p}$ for all $p\geq 0$. Applying Corollary~\ref{Thm:JumpNumber}, we conclude that every non-zero extreme point of $\Sigma_X^*$ has rank $1$.
\end{proof}

In the following section, we show that the equality $\Sigma_X^\ast=P_X^\ast$ implies $2$-regularity for a totally real reduced scheme $X$, which is the converse of Corollary~\ref{cor:reg} for reduced schemes. 
This will also allow us to classify all spectrahedra with only rank $1$ extreme rays.

\section{Hankel spectrahedra detect $2$-regularity}\label{Sec:TRreg2}
The main goal of this section is to prove the following theorem.
\begin{theorem}\label{thm:2regmain}
Let $X\subset \PP^n$ be a closed totally real reduced scheme such that $X\neq\PP^n$. We have $P_X^*=\Sigma_X^*$ if and only if $X$ is $2$-regular.
\end{theorem}
From this statement, we will deduce the classification of spectrahedral cones with only rank $1$ extreme rays in Subsection~\ref{ssec:rank1}.

One direction was already shown in Corollary~\ref{cor:reg} and it remains to show that the equality of the dual cones implies $2$-regularity. Before we begin the proof of Theorem~\ref{thm:2regmain} we observe the following immediate corollary, which is a generalization of the main result of \cite{BSV} from varieties to reduced schemes.

\begin{cor}
Let $X\subset \PP^n$ be a closed totally real reduced scheme such that $X\neq\PP^n$. Every non-negative quadric on $X$ is a sum of squares in $\RR[X]$ if and only if $X$ is $2$-regular.
\end{cor}
\begin{proof}
This immediately follows from Theorem~\ref{thm:2regmain} since the cones $P_X$ are $\Sigma_X$ are closed when $X$ is a totally real reduced scheme.
\end{proof}

For the rest of the section, let $X\subset\PP^n$ be a closed reduced scheme with defining ideal $I \subsetneq S$. 

We begin the proof of Theorem~\ref{thm:2regmain} by proving that if $\Sigma_X^*=P_X^*$, then this property is preserved by taking both hyperplane sections and linear projections. Projection away from a point $p\in \PP^n(\RR)$ determines a rational map $\pi_p: \PP^n\dashedrightarrow \PP^{n-1}$. This map extends to a morphism from the blow-up of $\PP^n$ at $p$, which we denote by $\hat{\pi}_p: {\rm Bl}_{p}(\PP^n)\rightarrow \PP^{n-1}$. Let $\pi_p(X)$ be the image  in $\PP^{n-1}$ via $\hat{\pi}_p$ of the strict transform of $X$ in ${\rm Bl}_{p}(\PP^n)$, which is equal to the Zariski closure of the image of $X$ under $\pi_p$. If $W\subseteq \CC^{n+1}$ is a real subspace such that $\CC^{n+1}=W\oplus \langle p\rangle$ is a direct sum decomposition with projection $\overline{\pi}: \CC^{n+1}\rightarrow W$, we can identify the image space $\PP^{n-1}$ of $\pi_p$ with $\PP(W)$. This identification gives an inclusion of the coordinate ring $S'$ of $\PP(W)$ into $S$ and the ideal of definition of $Y=\pi_p(X)$ in $\PP(W)$ is given by $I(Y)=I(X)\cap S'.$ In particular, there is an inclusion $\RR[Y]\rightarrow \RR[X]$, so the restriction map $S_2^*\rightarrow (S'_2)^*$ takes forms which annihilate $I(X)_2$ to forms which annihilate $I(Y)_2$. A point evaluation $\ev_{p'}\in P_X^*$ with $p'\neq p$ is mapped via this restriction to the point evaluation $\ev_{\pi_p(p')}\in P_Y^*$ and $\ev_p$ is mapped to zero.

\begin{lemma}\label{lem:HSandP}
Let $X\subset\PP^n$ be a closed reduced totally real scheme and suppose $\Sigma_X^*=P_X^*$. Then the following statements hold.
\begin{enumerate} 
\item The real locus of $X$ is cut out by quadrics, i.e.~$X(\RR)= \VV(I(X)_2)(\RR)$
\item For any real line $L\subseteq \PP^n$, either $X$ contains $L$ or $X\cap L$ is a zero-dimensional scheme of length at most two.
\item Let $h\in S_1$ be a nonzero linear form and let $Y\subseteq \PP^n$ be the reduced scheme supported on $X\cap \VV(h)$. Then the equality $\Sigma_Y^*=P_Y^*$ holds.
\item Suppose $X(\RR)$ is non-degenerate and pick $p\in X(\RR)$. Let $Y:=\hat{\pi}_{p}(X)\subseteq \PP^{n-1}$ be the image of the projection away from $p$. Then the equality $\Sigma_Y^*=P_Y^*$ holds. Moreover, if $p$ is a regular real point of $X$, then every point in $Y(\RR)$ has a preimage in ${\rm Bl}_{p}(X)(\RR)$. 
\end{enumerate} 
\end{lemma}
\begin{proof}
$(1)$ Suppose there is a real point $[\alpha]\in \VV(I(X)_2)\setminus X$, then the linear functional $\ev_\alpha$ lies in the cone $\Sigma_X^*$. It is not an element of $P_X^*$, because we can separate $\alpha$ from the compact set $X(\RR)$ by a quadric. 
$(2)$ Restricting $I(X)_2$ to a line $L$ and using part $(1)$ we see that either $L(\RR)\subseteq X(\RR)$ and thus $L\subseteq X$ or the intersection has at least a nonzero quadric in its ideal of definition in $\RR[L]$ and thus has length at most two. $(3)$ If $\ell\in S_2^*$ belongs to $\Sigma_Y^*$, the bilinear form $B_\ell\colon R_1\times R_1 \to \RR$, $(f,g)\mapsto \ell(fg)$, is positive semidefinite and $\ell$ annihilates the quadratic part of the ideal of definition of $Y$. In particular, $\ell$ annihilates $(I(X))_2$ and thus $\ell\in \Sigma_X^*$. Because $\Sigma_X^*=P_X^*$, there exist real numbers $c_i\geq 0$ and points $p_i\in X(\RR)$ such that $\ell=c_1\ev_{p_1}+\dots +c_s\ev_{p_s}$. Since $\ell(h^2)=0$, we conclude that $p_i\in Y(\RR)$ and $\Sigma_Y^*=P_Y^*$. 

$(4)$ We use the notation fixed in the paragraph before the statement of the lemma. Since the coordinate ring $\RR[Y]$ is contained in $\RR[X]$, every linear form $\ell' \in S'$ which annihilates $I(Y)_2$ can be extended to a linear form $\ell\in S$ which annihilates $I(X)_2$. If $B_{\ell'}$ is positive definite and $\ell$ is any such extension, then the form $\ell + a\ev_{p}$ annihilates $I(X)_2$ for any real number $a$. Moreover, the associated bilinear form is positive definite for all sufficiently large $a$. We choose $\mu$ to be one of these positive definite extensions. Since $\Sigma_X^*=P_X^*$, there exist real numbers $c_i$ and real points $p_1,\dots,p_m\in X(\RR)$ such that $\mu=c_1\ev_{p_1}+\dots + c_m\ev_{p_m}$. Because the restriction map takes $\ev_p$ to zero, we conclude that $\mu\in P_Y^*$. So every positive definite form in $\Sigma_Y^*$ belongs to $P_Y^*$. The set $Y(\RR)$ is non-degenerate because $X(\RR)$ is non-degenerate and thus the cone $\Sigma_Y^*$ contains at least one, and therefore a dense set, of positive definite forms. We conclude that $\Sigma_Y^*\subseteq P_Y^*$. The opposite inclusion is immediate, because convex duality reverses inclusion. This proves the first part of the claim. 

Now assume $p\in X(\RR)$ is a regular point on $X$. We will show that every real point $y\in \pi_p(X)$ has a real preimage in ${\rm Bl}_{p}(X)$. In case $y$ has a preimage in $X\setminus\{p\}$, consider the real line $L\subset\PP^n$ spanned by $p$ and $y$, which is real. So by part $(1)$ either $L\subseteq X$ or $L\cap X$ is a scheme of length at most two containing $p$ and some other point $x\in X$. In the first case, any point of $L(\RR)$ is a preimage for $y$ in $X(\RR)$. In the second case, the point $x$ must be real because otherwise its conjugate $\overline{x}$ also belongs to $L$ proving that $L$ has length at least three, a contradiction. If a real point $y\in \PP^{n-1}=\PP(W)$ is in the image of the exceptional divisor of the blow-up, then $y$ is the tangent direction of a real tangent line $L$ to $p$. Because $p$ is non-singular the direction of $L$ at $p$ is a limit of directions of real lines intersecting $X$ at $p$ and at another real point in $X\setminus \{p\}$
and therefore the tangent direction of $L$ at $p$ corresponds to a real point of ${\rm Bl}_{p}(X)$, as claimed.
\end{proof}

\begin{lemma}\label{lem:lindep}
Let $X\subset\PP^n$ be a closed reduced scheme and suppose $\Sigma_X^\ast = P_X^\ast$. If $X\subseteq \PP^n$ contains a set $B\subseteq X(\RR)$ of real points each of which is an irreducible component of $X$, then we have $\langle B \rangle \cap X = B$. In particular, if $B$ spans $\PP^n$, then we have $X=B$.
\end{lemma}
\begin{proof}
By Lemma~\ref{lem:HSandP}$(3)$, we can restrict to the scheme $X\cap\langle B\rangle$ in the subspace $\langle B\rangle$ spanned by $B$. We therefore assume that $B$ spans $\PP^n$ and prove that $X=B$ by induction on $n$. The statement is trivial if $n=0$ and holds for $n=1$ because $X(\RR)$ is cut out by quadrics, see Lemma~\ref{lem:HSandP}$(1)$.
If $n>1$ and $p\in B$, then by Lemma~\ref{lem:HSandP}$(4)$, the projection of $X$ away from $p$ produces a reduced scheme $Y\subseteq \PP^{n-1}$ with $\Sigma_Y^*=P_Y^*$. By Lemma~\ref{lem:HSandP}$(2)$ the projection away from $p$ maps the elements in $B\setminus p$ to a set of $0$-dimensional components of $Y$ which span $\PP^{n-1}$. By induction, $\pi_p(X)$ agrees with $Y$ and thus $X$ is contained in the inverse image of the projection of this set. We conclude that $X=B$ as claimed. 
\end{proof} 

\begin{remark} The assumption that $X$ contains a basis consisting of real points is necessary for the conclusion to hold, even for the case of points. Let $Q=x_0^2+\dots + x_n^2$ and let $X$ be a $0$-dimensional complete intersection defined by a set of quadrics containing $Q$. The scheme $X$ consists of $2^{n-1}$ conjugation invariant pairs of complex points and it properly contains a basis if $n>1$. However, $\Sigma_X^*=P_X^*=0$ because $Q\in I(X)$.\end{remark}

For the rest of this section, let $X = X_1\cup\dots\cup X_r$, where the $X_i$ are the (geometrically) irreducible components of $X$, and let $L_i:=\langle X_i\rangle$ be the span of $X_i$.

\begin{lemma}\label{lem:componentsMinDegree}
Let $X\subset\PP^n$ be a closed reduced totally real scheme and suppose $\Sigma_X^*=P_X^*$. Every irreducible component of $X$ is a variety of minimal degree in its span. Moreover, the span of any irreducible component $X_i\subset X$ intersects $X$ only in $X_i$, i.e.~$\left(X\cap L_i\right)=X_i$.
\end{lemma} 
\begin{proof}
Let $U$ be the projective space spanned by a set $B$ consisting of $\codim_{L_i}(X_i)+1$ general real points of $X_i$ and let $Y$ be the reduced scheme supported on $X\cap U$. Each $X_i(\RR)$ is Zariski-dense in $X_i$ by assumption and therefore non-degenerate. It follows that $Y$ contains a $0$-dimensional scheme of $\deg(X_i)$ many points by Bertini's Theorem because we can cover a neighborhood of $U$ in the Grassmannian by moving the chosen generic real points spanning $U$ in $X_i(\RR)$. By Lemma~\ref{lem:HSandP}$(3)$, we know that $\Sigma_Y^*=P_Y^*$. By Lemma~\ref{lem:lindep}, we conclude that $B=Y$ and in particular $\deg(X_i)=\codim_{L_i}(X_i)+1$; so $X_i$ is a variety of minimal degree in its span.

We now show that $(X\cap L_i)(\RR) = X_i(\RR)$, which will imply the claim. 
Suppose there is a real point $q$ in $(X\cap L_i)(\RR)$, which is not in $X_i$. Let $W$ be the projective linear span of $q$ and $\codim_{L_i}(X_i)$ general real points in $X_i$ so that $(X\cap W)(\RR)$ contains a real basis of $W$.
Then $W\cap X$ is $0$-dimensional because the dimension of every irreducible component of $X\cap L_i$ different from $X_i$ is smaller than $\dim(X_i)$ by the argument in the first paragraph of the proof. But $W$ intersects $X_i$ in another real point for degree reasons, which contradicts Lemma \ref{lem:lindep}.

So $(X\cap L_i)(\RR) = X_i(\RR)$, from which we conclude $\Sigma_Y^\ast = P_Y^\ast = P_{X_i}^\ast = \Sigma_{X_i}^\ast$ by Lemma \ref{lem:HSandP}(3). Therefore, the ideals of $X\cap L_i$ and $X_i$ are equal in the degree $2$ part of the coordinate ring of $L_i$. The scheme $X_i$ is defined by quadrics in $L_i$ because it is a variety of minimal degree. Therefore, we obtain $X\cap L_i = X_i$.
\end{proof}

Next we focus on projections of $X$ from a point. Let $p\in X_1(\RR)$ be a generic point, let $Z=\pi_p(X)$ and $Z_i: =\pi_p(X_i)$. The irreducible components of $Z$ are precisely the $Z_i$ which are maximal under inclusion. If $W=\bigcup_{i=1}^r W_i$ is the decomposition of a reduced scheme $W$ into its irreducible components and $1\leq j\leq r$ we let $W_j':=\bigcup_{j\neq i} W_j$.  The following Lemma allows us to understand the intersections of components of $X$ by looking at their projections.

\begin{lemma}\label{lem:pairwiseInt}
Let $X\subset\PP^n$ be a closed totally real reduced scheme and suppose $\Sigma_X^*=P_X^*$. Let $p$ be a generic point of $X_1(\RR)$. The following statements hold:
\begin{enumerate}
\item If $Z_i\subseteq Z_j$ with $i\neq j$, then $i=1$ and $X_1$ is a linear space.
\item The inclusion $ Z_i'\subseteq\pi_p(X_i')$ holds with equality unless $Z_1\subset Z_i$ and $Z_1\not\subset Z_j$ for all $j\neq i$.
\item If for some irreducible component $Z_i$ of $Z$ the intersection $Z_i\cap Z_i'$ is integral and totally real and $\pi_p(X_i')= Z_i'$, then $\pi_p(X_i\cap X_i')=Z_i\cap Z_i'$.
\end{enumerate}
\end{lemma}
\begin{proof}
$(1)$ Let $x_i\neq p$ be a general point of $X_i(\RR)$. If $j\neq 1$, then there is a point $x_j\in X_j$ with $\pi_p(x_i) = \pi_p(x_j)$ because $Z_i\subset Z_j$ and $\pi_p\vert_{X_j}\colon X_j\to Z_j$ is a morphism. Because $i\neq j$ and $x_i$ is generic, we can assume $x_i\neq x_j$. So by Lemma~\ref{lem:HSandP}$(2)$, we conclude that the real line $\Lambda$ spanned by $p$ and $x_i$ must be contained in $X$. Since $X_1$ is the only irreducible component of $X$ containing $p$, we must have $\Lambda\subset X_1$. But $x_i\in\Lambda$ is generic on $X_i$, which implies $X_i\subset X_1$, i.e.~$i=1$. Both points $x_i$ and $p$ are generic in $X_1$ and their span is in $X_1$, so $X_1$ is a linear space.
We will now show that the case $j=1$ cannot occur by contradiction. If $j$ were $1$, then either there exists $x_1$ with $\pi_p(x_i) = \pi_p(x_1)$ or there is a tangent line to $X_1$ at $p$, which maps to $\pi_p(x_i)$ under $\hat{\pi}_p\colon {\rm Bl}_p(X) \to Z$. In the first case, we find $i=j=1$ by the same argument as before, which contradicts $i\neq j$. In the latter case, the line spanned by $x_i$ and $p$ is the tangent line $L$ to $X_1$ at $p$, so if $L\cap X$ were $0$-dimensional, it would have length at least $3$, which contradicts Lemma \ref{lem:HSandP}$(2)$. So as before, we conclude $i=j=1$, which is a contradiction.

$(2)$ By definition, we have $\pi_p(X_i')=\bigcup_{j\neq i}\pi_p(X_j)=\bigcup_{j\neq i}Z_j$ while $Z_i'=\bigcup_{j\neq i: Z_j\text{ is maximal}} Z_j$. Suppose $Z_i'\subsetneq \bigcup_{j\neq i} Z_j$, then there is an index $s\neq i$ such that $Z_s$ is not maximal, i.e.~$Z_s\subset Z_i$.
From part $(1)$, we conclude that $s=1$.

$(3)$ The inclusion $\pi_p(X_i\cap X_i')\subseteq \pi_p(X_i)\cap \pi_p(X_i')=Z_i\cap Z_i'$ is immediate.
Let $z$ be a generic real point of $Z_i\cap Z_i'$. We will show that there are always $x_i\in X_i$ and $x_i'\in X_i'$ such that $\pi_p(x_i) = z = \pi_p(x_i')$: First, consider the case $i=1$. Then there exists $x_i'\in X_i'$ with $\pi_p(x_i')=z$, because $\pi_p$ is a morphism if restricted to $X_j$ for $j\neq i$. So the line $\Lambda = \pi_p^{-1}(z)$ intersects $X$ in $x_i'$ and $p$. As in the proof of part $(1)$, we conclude $\Lambda\subset X_1$, because there exists $x_i\in X_i(\RR)$ with $\pi_p(x_i) = z$ or $\Lambda$ is tangent to $X_i$ at $p$. So there exists $x_i\in X_i$ with $\pi_p(x_i) = z = \pi_p(x_i')$.
Next, we assume $i\neq 1$. Then there is a point $x_i\in X_i$ with $\pi_p(x_i) = z$. Again, we can use this point to show that the line $\pi_p^{-1}(z)$ is contained in $X_1$, so there is a point $x_i'\in X_1\subset X_i'$ with $\pi_p(x_i') = z = \pi_p(x_i)$.

In conclusion, we can always find real points $x_i\in X_i$ and $x_i'\in X_i'$ with $\pi_p(x_i) = z = \pi_p(x_i')$.
So the real line $\Lambda = \pi_p^{-1}(z)$ intersects $X$ at $p$ and points $x_i\in X_i$ and $x_i'\in X_i'$. By Lemma~\ref{lem:HSandP}$(2)$, $\Lambda$ intersects $X$ either in a zero-dimensional scheme of length at most two or $\Lambda\subseteq X$. In the first case we conclude that $x_i=x_i'\in (X_i\cap X_i')(\RR)$ and thus $z\in \pi_p(X_i\cap X_i')$. In the second case we conclude that $\Lambda\subseteq X_1$ because $X_1$ is the only component of $X$ containing $p$.
If $i\neq 1$, then $x_i\in X_i\cap X_1\subseteq X_i\cap X_i'$ and thus $z\in \pi_p(X_i\cap X_i')$. If $i=1$, then $x_i'\in X_1'\cap X_1$ and thus $z\in \pi_p(X_i\cap X_i')$ as claimed.
\end{proof}

\begin{remark} The inclusion $ Z_i'\subseteq\pi_p(X_i')$ in part $(2)$ of the previous lemma can be strict. As an example, let $X=L_1\cup L_2\cup L_3$ be the union of a line $L_2$ meeting two skew lines $L_1,L_3$ in $\PP^3$ and consider the projection away from a general point in $L_1$. Then $\pi_p(L_2') = Z_1\cup Z_3$, whereas $Z_2' = Z_3$.
\end{remark}

A scheme $X$ is linearly joined if there exists an ordering $X_i$, $i = 1,\dots,r$, of the irreducible components of $X$ such that the equality $X_j \cap \left(\bigcup_{s<j}X_s\right) = \langle X_j\rangle \cap\langle\bigcup_{s<j}X_s\rangle$ holds for $1<j\leq r$. 

\begin{Def}
We say that an irreducible component $X_j$ of $X$ is an \emph{end} of $X$ if the equality $X_j\cap X_j' =\langle X_j\rangle \cap \langle X_j'\rangle$ holds. 
\end{Def}
The terminology is motivated by the fact that the last irreducible component in any ordering that makes $X$ linearly joined. Moreover, if $X_j$ is an end, then $\langle X_j'\rangle \cap X = X_j'$. So ends can be used to construct the required orderings inductively.

\begin{lemma}\label{lem:Irreducible}
Assume $X\subseteq \PP^n$ is closed, reduced, and totally real such that $\Sigma_X^*=P_X^*$. If there is a real point $p\in X$ such that $\pi_p(X)$ is irreducible, then $X$ has at most two irreducible components which are varieties of minimal degree in their span. Furthermore, $X$ is linearly joined.
\end{lemma}

\begin{proof}
Since $Z =\pi_p(X)$ is irreducible, we conclude from Lemma~\ref{lem:pairwiseInt}$(1)$ that $X$ has at most $2$ irreducible components. 
If $X$ is irreducible, then $X$ is a variety of minimal degree in its span by Lemma~\ref{lem:componentsMinDegree}.
If $X$ has $2$ irreducible components, then one of them is a linear space by Lemma \ref{lem:pairwiseInt}(1), say $X_1$, and the center of the projection $p$ is in $X_1$.
Then $X_1\cap X_2$ is a hyperplane in the linear space $X_1$. Indeed, $X_1\cap X_2$ must have codimension $1$ and if $X_1\cap X_2$ were not a hyperplane in $X_1$, then $X_1\subseteq \langle X_2\rangle$, contradicting Lemma~\ref{lem:componentsMinDegree}.
\end{proof}

\begin{theorem}\label{thm:psdsossmall}
If $X\subseteq \PP^n$ is closed, reduced, and totally real such that $\Sigma_X^*=P_X^*$, then $X$ is linearly joined and its irreducible components are varieties of minimal degree in their span.
\end{theorem}
\begin{proof}
The proof is by induction on $n$. 
The case $n=0$ is immediate, the case $n=1$ follows from Lemma~\ref{lem:HSandP}$(2)$. For $n>1$, assume $X$ has $r$ irreducible components $X_1,\dots,X_r$. We will also use induction on $r$. The case where $X$ is integral follows from Lemma~\ref{lem:componentsMinDegree}. In case that $X$ is not integral, we will additionally show in the induction that $X$ has at least two distinct ends.

If $r>1$, let $p\in X_1(\RR)$ be a generic point and let $Z:=\pi_p(X)\subseteq \PP^{n-1}$. The variety $Z$ is totally real and satisfies $\Sigma_Z^*=P_Z^*$ by Lemma~\ref{lem:HSandP}$(4)$. Therefore, by the induction hypothesis on $n$, $Z$ is a linear join of varieties of minimal degree. 
If $Z$ is irreducible, then $X$ has two irreducible components and is linearly joined by Lemma~\ref{lem:Irreducible}. Both components are ends. If $Z$ has at least two irreducible components, let $i$ be an index such that $Z_i$ is an end for $Z$. To complete the induction step, we will use the fact that
if the equality $\pi_p(X_i')=Z_i'$ holds, $X_i$ is an end for $X$. We first prove this claim. Suppose for contradiction that the inclusion $X_i\cap X_i'\subset \langle X_i\rangle \cap \langle X_i'\rangle$ is strict.
Because $Z_i$ is an end of $Z$, the intersection $Z_i\cap Z_i'$ is integral and totally real. So by Lemma \ref{lem:pairwiseInt}$(3)$, we have the inclusions
\[ Z_i\cap Z_i' = \pi_p(X_i\cap X_i')\subseteq \pi_p\left(\langle X_i\rangle \cap \langle X_i'\rangle\right)\subseteq \pi_p\left(\langle X_i\rangle\right) \cap \pi_p\left( \langle X_i'\rangle\right) =  \langle Z_i\rangle \cap \langle Z_i'\rangle.
\]
Since $Z_i\cap Z_i'=\langle Z_i\rangle \cap \langle Z_i'\rangle$, all these inclusions are actually equalities. In particular, $\dim(\pi_p(X_i\cap X_i')) = \dim(\pi_p(\scp{X_i}\cap\scp{X_i'}))$, which shows that
the projection of $\langle X_i \rangle \cap \langle X_i' \rangle$ away from $p$ has smaller dimension than $\langle X_i\rangle \cap \langle X_i'\rangle$; hence we must have $p\in \langle X_i\rangle \cap \langle X_i'\rangle$. Since $p$ is a generic point of $X_1$, it follows that $X_1\subseteq \langle X_i\rangle \cap \langle X_i'\rangle$; so the index $i$ must equal one by Lemma~\ref{lem:componentsMinDegree}.
Applying $\pi_p$ to the inclusion $X_i\subseteq \langle X_i\rangle \cap \langle X_i'\rangle$ we conclude that $Z_i\subseteq Z_i\cap Z_i'$ and therefore $Z_i$ is not an irreducible component of $Z$, a contradiction. This contradiction proves the claim that $X_i$ is an end for $X$.

Next, we show that the irreducible component $X_1$ containing $p$ can be chosen to construct at least two distinct ends for $X$ using the claim. 
We distinguish two cases according to Lemma \ref{lem:pairwiseInt}$(2)$.
\begin{enumerate}
\item{There is an irreducible component $X_1$ such that the image $\pi_p(X_1) = Z_1$ is either an irreducible component of $\pi_p(X)$ or $Z_1$ is contained in at least two distinct irreducible components of $\pi_p(X)$.}
\item{Every irreducible component $X_1$ of $X$ has the property that after projection from a generic point $p\in X_1$, the image $\pi_p(X_1) = Z_1$ is contained in $\pi_p(X_j) = Z_j$ for exactly one $j\neq 1$.}
\end{enumerate}
In case $(1)$, Lemma~\ref{lem:pairwiseInt}$(2)$ shows that $\pi_p(X_i')=Z_i'$ for every irreducible component $Z_i$ of $Z$. It follows that if $Z_{i}$ and $Z_{j}$ are distinct ends of $Z$, then $X_{i}$ and $X_{j}$ are distinct ends of $X$ by the above claim.
In case $(2)$, Lemma~\ref{lem:pairwiseInt}$(2)$ shows that $\pi_p(X_i')=Z_i'$ for every irreducible component $Z_i$ of $Z$ except for the unique irreducible component $Z_k$ of $Z$ which contains $Z_1$. Since $Z$ has at least two ends, at least one of them, say $Z_m$, is not $Z_k$. As before, the above claim shows that $X_m$ is an end for $X$. Repeating the argument for the projection away from a general point in $p'\in X_m$, we produce a second end for $X$ which must necessarily be distinct from $X_m$ because $\pi_{p'}(X_m)$ is not an irreducible component of $\pi_{p'}(X)$ since we are in case $(2)$.
We conclude that, in all cases, $X$ has at least two distinct ends. If $X_i$ is one of them, $X\cap \langle X_i'\rangle =X_i'$ and by Lemma~\ref{lem:HSandP}$(3)$ and the induction hypothesis on $r$ we conclude that $X_i'$ is a linear join of varieties of minimal degree. As a result $X=X_i'\cup X_i$ is a linear join of varieties of minimal degree by putting $X_i$ as the last component in the ordering, proving the Theorem.
\end{proof}

\begin{proof}[Proof of Theorem~\ref{thm:2regmain}]
Finally, it is easy to show that a scheme which is a linear join of varieties of minimal degree must have regularity $2$, see~\cite[Proposition $3.1$]{EGHP}. So Theorem~\ref{thm:2regmain} follows from Theorem \ref{thm:psdsossmall}.
\end{proof}

\subsection{Classification of spectrahedral cones with only rank $1$ extreme rays.}\label{ssec:rank1}
We now show that the results of the previous section allow us to classify spectrahedral cones whose extreme rays all have rank $1$. Let $K=S_+\cap L$ be a spectrahedral cone. If $L$ does not pass through the interior of $S_+$, then let $F$ be the largest face of $S_+$ such that $L$ passes through the relative interior of $F$. It follows that $L$ is contained in the span $\langle F \rangle$ of $F$. By~\cite[Theorem 1]{RamanaGoldman}, $F$ is the cone of positive semidefinite quadratic forms of a proper subspace $S_1$. Therefore it suffices to characterize all spectrahedral cones $K=S_+\cap L$ with only rank $1$ extreme rays where $L$ passes through the interior of $S_+$.

\begin{theorem}
Let $K=S_+\cap L\subsetneq S_+$ be a spectrahedral cone such that every extreme ray of $K$ has rank $1$, and suppose that the linear subspace $L$ passes through the interior of $S_+$. Then $K$ is the Hankel spectrahedron of a non-degenerate, reduced, $2$-regular, totally real scheme $X$ and $L=(I(X)_2)^\perp$.
\end{theorem}
\begin{proof}
Let $I$ be the smallest radical ideal containing $L^\perp$, the subspace of $S_2$ of linear functionals vanishing on $I$. Then $I$ defines a reduced scheme $X=\VV(I)\subset \PP^n$. Since $L$ passes through the interior of $S_+$, $X$ is non-degenerate. Since $K$ is full-dimensional in $L$ and generated as a convex cone by quadratic forms of rank $1$, Lemma~\ref{Lem:spectrahedral} implies that $X$ is totally real, $K$ is the Hankel spectrahedron of $X$, and $\Sigma_X^*=P_X^*$. The Theorem now follows from Theorem \ref{thm:2regmain}.
\end{proof}

\section{Sharpness of Hankel index inequalities.}\label{SEC:BPF}
This section exhibits several classes of varieties for which the Hankel index equals the Green-Lazarsfeld index plus one (i.e. where the inequality in Theorem~\ref{Thm:JumpNumberpBp} is an equality). More generally, we introduce tools which can be used to establish this equality, obtaining a method to effectively compute the Hankel index from the free resolution via, for instance, Gr\"obner basis computations. We begin by introducing key numerical quantities of varieties $X$ over $\CC$.

Let $\alpha(X)$ be the largest integer $p\geq 1$ such that $X$ satisfies property $N_{2,p}$, i.e.~$\alpha(X)$ is the Green-Lazarsfeld index of $X$. Furthermore, let $\beta(X)$ be the largest integer such that $X$ has the $p$-base-point-free property. And lastly, we write $\gamma(X)$ for the largest $p\geq 1$ such that $X$ is $p$-small.
By Theorem~\ref{thm:pBp}, the inequalities $\alpha(X)\leq \beta(X)\leq \gamma(X)$ hold.
By Theorem~\ref{Thm:JumpNumberpBp}, the Hankel index $\eta(X)$ is bounded below by the inequality $\beta(X)+1\leq \eta(X)$.
Our next theorem is a useful tool for establishing upper bounds on the Hankel index whenever the failure of $p$-smallness can be witnessed by sets of real points. It provides an explicit construction of linear functionals $\tau\in \Sigma_X^*\setminus P_X^*$ which follows~\cite[Proposition 6.2]{BleNP}.

\begin{theorem}\label{Thm: deathRay} Let $X\subseteq \PP^n$ be a reduced and non-degenerate scheme which is set-theoretically defined by quadrics. If $W\subseteq \PP^n$ is a projective subspace of dimension $p$ such that $\Gamma:= W\cap X$ is zero-dimensional and contains $p+2$ real reduced points, then we have $\eta(X)\leq p$.
\end{theorem}

\begin{proof}
Let $\Gamma'\subseteq \Gamma$ be a minimal linearly dependent subset and let $W':=\langle \Gamma\rangle$. By construction $W'$ is a subspace of some dimension $k$ with $1\leq k\leq p$ containing the set $\Gamma'$ consisting of $k+2$ points $q_1,\dots q_{k+2}$ in linearly general position in $W'$.
Because $X$ is defined by quadrics and $\Gamma$ is finite we know that $k\geq 2$.

By minimality of $\Gamma'$, there is a linear relation $\sum_{i=1}^{k+2} u_i \ev_{q_i}=0$ among the functionals $ev_{q_i}$ in the coordinate ring of $W'$ which is unique, up to multiplication by a scalar and all its coefficients $u_i$ are nonzero real numbers.
To prove the claim, we show that there exist real numbers $a_i$ such that the linear functional $\tau:=\sum_{i=1}^{k+2}a_i\ev_{q_i}$ is an element of $\Sigma_X^*\setminus P_X^*$ with rank equal to $k$. 

If $g$ is a linear form, then we have $\tau(g^2)=\sum_{i=1}^{k+2}a_i g(q_i)^2$. Using the fact that 
$\sum_{i=1}^{k+2} u_i\ev_{q_i}=0$, we obtain $\tau(g^2)=Q\left(g(q_1),\dots, g(q_{k+1})\right)$, where $Q$ is the quadratic form in $\RR^{k+1}$ given by
\[Q(y_1,\dots, y_{k+1})= \sum_{i=1}^{k+1}a_iy_i^2+a_{k+2}\left(\frac{\sum_{i=1}^{k+1} u_iy_i}{-u_{k+2}}\right)^2.\]
For any choice of positive real numbers $a_1,\dots,a_{k+1}$, set $a_{k+2}$ to be the negative reciprocal of the maximum of the square $f = \left(\sum_{i=1}^{k+1} -u_i/u_{k+2} y_i\right)^2$ on the compact ellipsoid $E$ defined by $\sum_{i=1}^{k+1} a_iy_i^2 = 1$. This choice makes the quadratic form $Q$ positive semidefinite with a zero at the maximizer of the square $f$ on $E$.
The quadratic form $Q$ is strictly positive in the subspace $\sum u_iy_i=0$, so that its rank is exactly $k$.
In conclusion, the linear form $\tau\in \Sigma_X^*$ is nonnegative on squares and the kernel $\Ker(B_{\tau})$ of its moment matrix is generated by $(W')^{\perp}$ and a linear form $\ell$, which is uniquely determined by the choice of $a_1,\dots,a_{k+1}$, in the following way. Let $(v_1,\dots,v_{k+1})$ be the maximizer of $f$ on $E$, then the conditions $\ev_{q_i}(\ell)= v_i$ for $1\leq i\leq k+1$ define a unique linear form $\ell$ because $q_1,\dots,q_{k+1}$ are linearly independent and $\tau(\ell^2) = 0$. For all sufficiently general choices of $a_1,\dots,a_{k+1}$, the linear form $\ell$ does not vanish at any point of $\Gamma$.
To finish the proof, we verify that any such $\tau$ is not in $P_X^*$. Suppose for contradiction that $\tau\in P_X^*$ so that $\tau$ is a convex combination of point evaluations. Since $\Ker(B_{\tau})\supseteq (W')^{\perp}$, it must to be a convex combination of real points in $\Gamma$. But then the linear form $\ell$ must vanish at some point of $\Gamma$ by $\tau(\ell^2) = 0$, which contradicts our choice of $a_1,\dots,a_{k+1}$ . So $\tau$ certifies that $\eta(X)\leq k$.
\end{proof}

\begin{remark} Arguing as in~\cite[Section $7$]{BleNP}, the construction above can be extended to the case when the failure of $p$-smallness is witnessed by a set $\Gamma$ which contains at most one pair of complex conjugate points.
\end{remark}

By a curve, we mean a complete integral scheme of dimension one. 
If $X$ is a curve over $\CC$, the gonality of $X$, denoted ${\rm gon}(X)$, is the smallest integer $r$ such that $X$ admits a non-constant morphism $f:X\rightarrow \PP^1$ of degree $r$. If $X$ is a real algebraic curve, we let ${\rm gon}(X)$ be the gonality of its complexification $X_{\CC}:=X\times_{\rm Spec(\RR)} {\rm Spec}(\CC)$. The Clifford index of a line bundle $\mathcal{L}$ on $X$ is given by ${\rm Cliff}(\mathcal{L}):=g+1-h^0(X,\mathcal{L})-h^1(X,\mathcal{L})$. If $X$ has genus $g\geq 4$, the Clifford index of $X$ is defined by ${\rm Cliff}(X):=\min\{{\rm Cliff}(\mathcal{L}): h^0(X,\mathcal{L})\geq 2, h^1(X,\mathcal{L})\geq 2\}$.

\begin{lemma}\label{lem:Bounds} Let $X\subseteq \PP^n$ be a reduced, non-degenerate scheme. The equalities $\alpha(X)=\beta(X)=\gamma(X)$ hold if either
\begin{enumerate}
\item $\gamma(X)\geq \codim(X)$, or
\item $\alpha(X)= \codim(X)-1$ and $X$ is not a hypersurface, or
\item $X$ is the canonical model of a general curve of genus $g\geq 4$.
\end{enumerate}
Moreover, we have $\beta(X)=\infty$ in case (1), $\beta(X) = \codim(X)-1 = \alpha(X)$ in case (2), and $\beta(X) = \lceil\frac{1}{2}(g-2)\rceil-1$ in case (3).
\end{lemma}
\begin{proof}
If $\gamma(X)\geq \codim(X)$, then $X$ is $p$-small for all $p\leq \codim(X)$, i.e.~$X$ intersects every linear subspace $L$ for which $X\cap L$ finite in a linearly independent set. Thus $X$ is small and $\alpha(X)=\beta(X)=\gamma(X)=\infty$ by Theorem~\ref{thm:pBp}, proving (1). 

If $\alpha(X)= \codim(X)-1$, then $\gamma(X)$ is $\codim(X)-1$ or $\gamma(X) = \infty$ by part (1) because $\alpha(X)\leq \gamma(X)$ by Theorem~\ref{thm:pBp}. When $\gamma(X) = \codim(X)-1 = \alpha(X)$, we get $\alpha(X)=\beta(X)=\gamma(X)=\codim(X)-1$ by Theorem~\ref{thm:pBp}. Since $\gamma(X) = \codim(X)$ cannot occur here by (1), this proves (2). 

Now suppose that $X$ is a non-singular curve of gonality $r$ which is not hyperelliptic and let $f:X\rightarrow \PP^1$ be a non-constant morphism of degree $r$. Let $L=f^{*}(s)$ be the divisor of a fiber of $f$ and let $X\subseteq \PP^{g-1}$ be the canonical model of $X$. We will show that $\gamma(X)\leq r-3$ by constructing a set of $r$ points spanning a projective subspace of dimension $r-2$. We claim that $h^0(X,\Osh_X[L])\leq 2$ since otherwise choosing an additional point $q$ of $X$ we could construct a divisor $L-q$ with at least two sections and thus a non-constant morphism $f':X\rightarrow\PP^1$ of degree strictly less than $r$. We conclude that $h^0(X,\Osh_X[L])=2$. By the Riemann-Roch Theorem
$2-h^1(X,\Osh_X[L]) = r+1-g$ so $h^1(X,\Osh_X[L])=g+1-r$ and by Serre duality $h^1(X,\Osh_X[L])=h^0(X,\Osh_X[K-L])$. 
In more geometric terms, we have proved that there exist exactly $g+1-r$ linearly independent forms in $\PP^{g-1}$ which vanish on the $r$ points of $L$ in $X$ so these $r$ points span a projective subspace of dimension $g-1-(g+1-r)=r-2$ in $\PP^{g-1}$ as claimed. We conclude that $\gamma(X)\leq r-3$. 

By the Brill--Noether Theorem~\cite[Theorem 8.16 and ensuing remark]{EisMR2103875}, the gonality of a general curve $X$ equals $r=\lceil\frac{1}{2}(g+2)\rceil$ so we conclude that $\gamma(X)\leq \lceil \frac{1}{2}(g-2)\rceil-1$. On the other hand Green's conjecture claims that for every non-hyperelliptic curve $X$ of genus $g\geq 4$ its canonical model should satisfy the equality $\alpha(X)= {\rm Cliff(X)}-1$. 
If $X$ is a general curve, then Green's conjecture is known to hold for $X$ by work of Voisin~\cite{V1},~\cite{V2} and Teixidor I Bigas~\cite{TIB}. By sharpness of the Brill-Noether Theorem~\cite{EH} we also know that ${\rm Cliff}(X)=\lceil \frac{1}{2}(g-1)\rceil$ for a general curve. Combining this fact with the inequalities of Theorem~\ref{thm:pBp} we conclude that $\alpha(X)=\beta(X)=\gamma(X)$ as claimed.
\end{proof}

\begin{remark} A reduced scheme $X$ satisfies $\gamma(X)\geq \codim(X)$ if and only if $X$ is small. If $X$ is irreducible, satisfies $\alpha(X)=\codim(X)-1$, and is not a hypersurface, then~\cite[Theorem 3.14]{HanKwak} shows that $X$ is an arithmetically Cohen-Macaulay variety of almost minimal degree (i.e. $\deg(X)=\codim(X)+2$). We exclude the case of hypersurfaces because they do not have a linear strand in their minimal free resolution.
\end{remark}

\begin{remark} The proof of the previous theorem shows that for every non-hyperelliptic curve $X$ of gonality $r$, the inequality $\gamma(X)\leq r-3$ holds. 
\end{remark}

\begin{remark} We restrict to curves of genus $g\geq 4$ because the only non-hyperelliptic curves of genus $g\leq 3$ are plane quartics and therefore their ideal is not generated by quadrics.
\end{remark}

\begin{remark}
It is natural to ask whether weakening condition $(2)$ above to $\gamma(X)= \codim(X)-1$ is a sufficient condition for the conclusion to hold. This is not the case. If $n\geq 1$ and $X$ is a set of $\binom{n+1}{2}$ points in general position in $\PP^n$, then $X$ imposes independent conditions on quadrics. It is immediate that the defining ideal of $X$ requires at least one cubic generator which implies $\alpha(X)=-\infty$ while $\gamma(X)=\codim(X)-1$. 
\end{remark}

Next we compute the Hankel index for all varieties appearing in Lemma~\ref{lem:Bounds} under additional arithmetic hypotheses. In all these, we prove that the lower bound from Corollary~\ref{Thm:JumpNumber} is sharp by constructing an element of $\ell\in\Sigma_X^*\setminus P_X^*$ of the appropriate rank via Theorem~\ref{Thm: deathRay}.

\begin{theorem}\label{Thm:MainJN} Let $X\subseteq \PP^n$ be a totally real, non-degenerate, reduced scheme. The equality $\eta(X)=\beta(X)+1$ holds whenever
\begin{enumerate}
\item $\gamma(X)\geq \codim(X)$ or
\item $\alpha(X) = \codim(X)-1$, $X$ is irreducible and is not a hypersurface or
\item {$X$ is the canonical model of a general curve $C$ of genus $g \geq 4$ whose gonality is totally real, in the sense that there exists a non-constant morphism $f:C\rightarrow \PP^1$ over $\RR$ of degree ${\rm gon}(X)$ with at least one fiber consisting of ${\rm gon}(C)$ real points.
}
\end{enumerate}
Moreover, $\eta(X)$ equals $\infty$, $\codim(X)$, and $\lceil\frac{1}{2}(g-2)\rceil$ respectively.
\end{theorem}
\begin{proof} By Theorem~\ref{Thm:JumpNumberpBp} the inequality $\beta(X)+1\leq \eta(X)$ holds for every $X$ so we obtain lower bounds for $\eta(X)$ from Lemma~\ref{lem:Bounds}.  
If $X$ satisfies $(1)$, then $\beta(X)=\infty$ by Lemma~\ref{lem:Bounds} and thus $\eta(X)=\infty$. 
If $X$ satisfies $\alpha(X)=\codim(X)-1$ and $X$ is not a hypersurface, then~\cite[Theorem 3.14]{HanKwak} shows that $X$ is an arithmetically Cohen-Macaulay variety  with $\deg(X)=\codim(X)+2$. Since $X$ is totally real, a generic set of $\codim(X)+1$ points spans a complementary subspace which intersects $X$ in $\codim(X)+2$ points. Since at least $\codim(X)+1$ of these are real, so is the last one. We conclude that $\eta(X)\leq \codim(X)$ by Theorem~\ref{Thm: deathRay}, proving the equality.
Finally, assume that $X\subseteq \PP^{g-1}$ satisfies $(3)$, let $r={\rm gon}(C)$ and let $q_1,\dots, q_r$ be the real distinct points of a fiber of $f$. Arguing as in the proof of Lemma~\ref{lem:Bounds} part $(3)$ we know that $r=\lceil\frac{1}{2}(g+2) \rceil$ and that these points span a projective space of dimension $m=r-2$ in $\PP^{g-1}$. By Theorem~\ref{Thm: deathRay} we conclude that $\eta(X)\leq \lceil\frac{1}{2}(g-2) \rceil$, proving the equality.
\end{proof}

\begin{remark} Any element $\ell\in \Sigma_X^*\setminus P_X^*$ with rank equal to $\eta(X)$ is automatically an extreme ray of $\Sigma_X^*$. In particular the elements constructed via Theorem~\ref{Thm: deathRay} in the proofs of part $(2)$ and $(3)$ of the previous theorem are extreme rays of $\Sigma_X^*$ which are not point evaluations.
\end{remark}

\section{Linear Joining and Sums of Squares}
In this section, we will give an alternative proof of one direction of one of our main theorems, namely that every nonnegative quadric on a reduced $2$-regular totally real scheme $X\subset \PP^n$ is a sum of squares. This proof has the advantage that we can keep track of the number of squares needed to represent a general nonnegative quadric as a sum of squares. This point of view gives a convex geometric description of the Hankel spectrahedron of a linearly joined scheme $X\cup Y$ in terms of the Hankel spectrahedra of $X$ and $Y$.

The following interpretation of sums of squares is the coordinate ring $\RR[X]_2$ of a real scheme $X\subset\PP^n$ follows from diagonalization of quadratic forms.

\begin{rem}\label{rem:rank}
Let $X \subseteq \PP^n$ be a scheme defined by $I\subseteq S$ and let $\RR[X]:=S/I$ be its homogeneous coordinate ring. Then a quadric $q\in\RR[X]_2$ is a sum of squares of linear forms in $\RR[X]$ if and only if there is a positive semidefinite quadric $Q\in\RR[x_0,\dots,x_n]_2$ such that $q + I_2 = Q + I_2\in \RR[X]_2$. Moreover, $q$ is a sum of at most $r$ squares in $\RR[X]_2$ if and only if there exists a positive semidefinite quadric $Q\in\RR[x_0,\dots,x_n]_2$ of rank at most $r$ representing $q$.
\end{rem}

Since every reduced $2$-regular scheme is a linearly joined sequence of varieties of minimal degree by Eisenbud-Green-Hulek-Popescu \cite[Theorem 0.4]{EGHP}, the following theorem will prove that every nonnegative quadric on a reduced totally real $2$-regular scheme is a sum of squares. The proof is similar to \cite[Lemma 2.5]{LVMR3207690}.

\begin{thm}\label{thm:linearjoins}
Let $X,Y\subset\PP^n$ be two real subschemes and suppose that $X\cap Y = \Span(X)\cap \Span(Y)$. Let $q\in \RR[X\cup Y]_2$ be a quadric and assume $f\vert_X$ and $f\vert_Y$ are both a sum of at most $r$ squares of linear forms in $\RR[X]_2$ and $\RR[Y]_2$, respectively. Then $f$ is a sum of at most $r$ squares of linear forms in $\RR[X\cup Y]_2$.
\end{thm}

\begin{proof}
As explained above in Remark \ref{rem:rank}, the fact that $f\vert_X$ is a sum of at most $r$ squares in $\RR[X]_2$ means that there is a positive semidefinite quadratic form of rank at most $r$ on $\Span(X)\subset\PP^n$, which is equal to $f\vert_X$ when restricted to $X$. Similarly, $f\vert_Y$ is equal to the restriction of a quadratic form of rank at most $r$ on $\Span(Y)\subset\PP^n$ to $Y$. We choose a basis of the real linear space $\Span(X)\cap \Span(Y)$ and extend it to a basis of $\PP^n$ by choosing complements of it in $\Span(X)$ and $\Span(Y)$, respectively. In these coordinates, we can express the fact that both $f\vert_X$ and $f\vert_Y$ have positive semidefinite extensions of rank at most $r$, which have to be compatible along $\Span(X)\cap \Span(Y)$, in terms of Cholesky factorizations. We write $\scp{v,w}$ for the standard inner product of two vectors $v,w\in\RR^r$.
There are vectors $p_1,\dots,p_d$ and $q_1,\dots,q_e$ in $\RR^r$ such that the $(n+1)\times (n+1)$ symmetric matrix
\[
\left(\begin{array}[h]{cc|ccc|cc}
\scp{p_1,p_1} & \dots & \scp{p_1,p_{d-k+1}} & \dots & \scp{p_1,p_d} & \ast & \ast \\
\vdots & & \vdots & & \vdots & \ast & \ast \\ \hline
\scp{p_{d-k+1},p_1} & \dots & \scp{p_{d-k+1},p_{d-k+1}} = \scp{q_1,q_1} & \dots & \scp{p_{d-k+1},p_d} = \scp{q_1,q_k} & \dots & \scp{q_1,q_e} \\
\vdots & & \vdots &  & \vdots & & \vdots \\
\scp{p_d,p_1} & \dots & \scp{p_d,p_{d-k+1}} = \scp{q_k,q_1} & \dots & \scp{p_d,p_d} = \scp{q_k,q_k}& \dots & \scp{q_k,q_e} \\ \hline 
 \ast & \ast & \vdots & & \vdots & & \vdots \\
 \ast & \ast & \scp{q_e,q_1} & \dots & \scp{q_e,q_k} & \dots & \scp{q_e,q_e} \\
\end{array}\right)
\]
corresponds to a quadratic form on $\PP^n$ whose restrictions to $X$ and $Y$, respectively, represent $f\vert_X\in\RR[X]_2$ and $f\vert_Y\in\RR[Y]_2$, respectively. The restrictions of this quadratic form to $X$ and $Y$ impose no conditions on the entries in the lower left and upper right block, which is why we denoted them by $\ast$.
The middle $k\times k$ block of the matrix is completely determined because it represents the restriction $f\vert_{X\cap Y}$ to $X\cap Y = \Span(X)\cap \Span(Y)$, a real linear space of dimension denoted by $k$.

So there exists an orthogonal change of coordinates $T\in O(r)$ such that $T(q_i) = p_{d-k+i}$ for all $i = 1,\dots k$. Therefore, the matrix
\[
\left(\begin{array}[h]{cc|ccc|cc}
\scp{p_1,p_1} & \dots & \scp{p_1,p_{d-k+1}} & \dots & \scp{p_1,p_d} & \ast & \ast \\
\vdots & & \vdots & & \vdots & \ast & \ast \\ \hline
\scp{p_{d-k+1},p_1} & \dots & \scp{p_{d-k+1},p_{d-k+1}} = \scp{Tq_1,Tq_1} & \dots & \scp{p_{d-k+1},p_d} = \scp{Tq_1,Tq_k} & \dots & \scp{Tq_1,Tq_e} \\
\vdots & & \vdots &  & \vdots & & \vdots \\
\scp{p_d,p_1} & \dots & \scp{p_d,p_{d-k+1}} = \scp{Tq_k,Tq_1} & \dots & \scp{p_d,p_d} = \scp{Tq_k,Tq_k}& \dots & \scp{Tq_k,Tq_e} \\ \hline 
 \ast & \ast & \vdots & & \vdots & & \vdots \\
 \ast & \ast & \scp{Tq_e,Tq_1} & \dots & \scp{Tq_e,Tq_k} & \dots & \scp{Tq_e,Tq_e} \\
\end{array}\right)
\]
also represents the restrictions of $f$ to $X$ and $Y$. It is also positive semidefinite of rank $r$, which proves the claim.
\end{proof}

Using this result, we can control the number of squares needed to represent a general nonnegative quadratic form on a small scheme.
\begin{cor}
Let $X\subset\PP^n$ be a totally real subscheme of regularity $2$. Then every quadratic nonnegative on $X$ is a sum of $\dim(X) + 1$ squares of linear forms in $\RR[X]_2$.
\end{cor}

\begin{proof}
This follows from Theorem \ref{thm:linearjoins} above and the count for varieties of minimal degree from \cite{BPSV} because a small scheme is the linear join of varieties of minimal degree by Eisenbud-Green-Hulek-Popescu \cite[Theorem 0.4]{EGHP}.
\end{proof}

From this point of view, we get a description of the Hankel spectrahedron of a linearly joined scheme $X\cup Y$ in terms of the Hankel spectrahedra of $X$ and $Y$. In particular, it implies that the extreme rays of the Hankel spectrahedron of $X\cup Y$ is the union of the extreme rays of the Hankel spectrahedra of $X$ and $Y$. The following result generalizes \cite[Theorem~3.1]{HeltonExtremePSD}.

\begin{cor}
Let $X,Y\subset\PP^n$ be two real subschemes and assume that $X\cap Y = \Span(X)\cap \Span(Y)$. Then the Hankel spectrahedron $\Sigma_{X\cup Y}^\ast$ is the convex hull of $\Sigma_{X}^\ast \cup \Sigma_Y^\ast$ and both of these Hankel spectrahedra are faces of $\Sigma_{X\cup Y}^\ast$.
\end{cor}

\begin{proof}
This follows from duality in convex geometry. We have the maps of real vector spaces
\includegraphics[width = \textwidth]{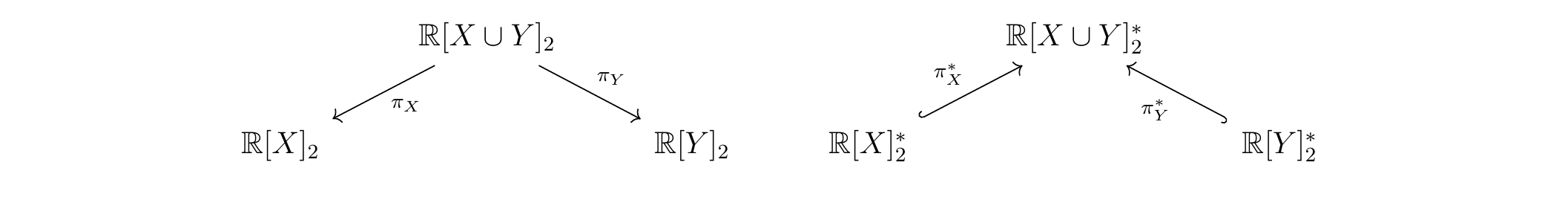}
The dual convex cone to $\pi_X^\ast(\Sigma_X^\ast)$ is the cone
\[
\{ q\in\RR[X\cup Y]_2 \colon \; \pi_X(q)\in \Sigma_X \} = \pi_X^{-1}(\Sigma_X).
\]
So by general duality in convexity, the claim that $\Sigma_{X\cup Y}^\ast$ is the conic hull of $\Sigma_X^\ast\cup \Sigma_Y^\ast$ is equivalent to the statement $\Sigma_{X\cup Y} = \pi_X^{-1}(\Sigma_X)\cap \pi_Y^{-1}(\Sigma_Y)$, which is proved in Theorem \ref{thm:linearjoins}. The fact that $\Sigma_X^\ast$ is a face of $\Sigma_{X\cup Y}^\ast$ then follows because it can be explicitly exposed by a sum of squares on $Y$, whose restriction to $\Span(X)\cap \Span(Y)=X\cap Y$ is $0$ and whose rank is equal to $\dim(\Span(Y)) - \dim(X\cap Y)$.
\end{proof}

The linear joining of schemes translates on the spectrahedral side to intertwining of cones introduced by Hildebrand \cite{Hildebrand}.

\section{Applications}

In this section, we relate our results to positive semidefinite matrix completion and the truncated moment problem.

\subsection{Monomial Ideals and Matrix Completion}\label{MatrixCompl}
 We will fix a simple graph $G=([n],E)$ on $n$ vertices.
This graph encodes a coordinate projection $\pi_G$ on the vector space $S^n$ of real symmetric $n\times n$ matrices 
\[
\pi_G\ \colon \; S^n \to \RR^n \oplus \RR^E, \  \pi_G(a_{ij}) = (a_{ii}\ \colon \; i\in [n]) \oplus (a_{ij}\ \colon \; \{i,j\}\in E).
\]
In other words, we project a matrix on its diagonal and its entries indexed by edges in the graph. We think of the image of a matrix under this projection $\pi_G$ as a partially specified matrix, i.e.~we get to complete it in the entries indexed by non-edges of the graph.
The semidefinite matrix completion problem asks which partial matrices can be completed to positive semidefinite matrices. Geometrically, we want to understand the image of the cone of {\psd} matrices under the projection $\pi_G$ in terms of the graph $G$.
In a more refined version of this completion problem, we might also ask for {\psd} completions satisfying additional rank constraints.

An obvious necessary condition for $\pi_G(a_{ij})$ to be in the image of the {\psd} cone is that all completely specified symmetric submatrices in the partial matrix are {\psd} (Paulsen-Power-Smith call this property \emph{partially positive} in \cite{PPSMR1005860}). 

\begin{exm}
Let $C_4 = ([4],\{ \{1,2\}, \{2,3\},\{3,4\},\{4,1\}\})$ be the four cycle. Then we would like to complete matrices of the form
\[
\begin{pmatrix}
a_{11} & a_{12} & \ast   & a_{14} \\
a_{12} & a_{22} & a_{23} & \ast   \\
\ast   & a_{23} & a_{33} & a_{34} \\
a_{14} & \ast   & a_{34} & a_{44}
\end{pmatrix},
\]
where the entries $a_{ij}$ are given and we get to choose the entries marked with a $\ast$. The completely specified symmetric sub-matrices correspond to the cliques in the graph, i.e.~we have four completely specified $2\times 2$ symmetric matrices (and, of course, the diagonal entries, which are the specified $1\times 1$ submatrices).
\end{exm}

This particular matrix completion problem comes up in Gaussian graphical models in statistics, where the image of the cone of positive semidefinite matrices is known as the cone of sufficient statistics of the Gaussian graphical model \cite[section 3]{UhlerMR3014306}.

We view the image of the cone of positive semidefinite matrices under this coordinate projection as the cone of sums of squares on the scheme defined by the Stanley-Reisner ideal of the clique complex of the graph.

The graph defines a square-free monomial ideal $I_G\subset \RR[x_1,\dots,x_n]$ generated by quadrics, namely $I_G = \langle x_ix_j\colon \{i,j\}\not\in E\rangle$. The monomial $x_i x_j$ is in $I_G$ whenever $\{i,j\}$ is not an edge of $G$. The ideal $I_G$ is the Stanley-Reisner ideal of the clique complex of the graph $G$.
In our notation, $I_G$ is the edge ideal of the dual graph of $G$, which is the common convention in the literature on positive semidefinite matrix completion.

The subscheme $X_G$ of $\PP^{n-1}$ defined by $I_G$ is the union of all coordinate subspaces
\[
\Span(\{e_i \colon i\in K\})
\]
over the cliques (complete induced subgraphs) $K$ in $G$. The dimension of $X_G$ is therefore the clique number of $G$ minus $1$. The following observation translates the matrix completion problem into sums of squares on a subspace arrangement.

\begin{lemma}\label{lem:tautology}
The map of homogeneous coordinate rings from $\RR[x_1,\dots,x_n]$ to $\RR[X_G]$ restricted to degree $2$ is the coordinate projection on the space of symmetric matrices encoded by the graph as described above. The image of the cone of positive semidefinite matrices under this projection is the cone of sums of squares on the subspace arrangement $X_G$.
Furthermore, the restriction of a quadratic form represented by $A$ to an irreducible component $U_K =\Span(\{e_i \colon i\in K\})$ of $X_G$ determines the completely specified submatrix of $A$ indexed by the maximal clique $K$ of $G$ corresponding to $U_K$.
\end{lemma}

\begin{proof}
This is immediate, see Remark \ref{rem:rank}.
\end{proof}

Since every sum of squares is nonnegative on $X_G(\RR)$, this gives an obvious necessary condition on the image of the {\psd} cone, namely $\Sigma_{X_G}\subset P_{X_G}$. Concretely, this recovers the obvious necessary condition that every completely specified submatrix is {\psd} mentioned before.

Regularity of the arrangement $X_G$ of subspaces can be described in terms of the graph $G$. The first result in this direction, that we want to mention, is due to Fr\"oberg. A graph is \emph{chordal} if every cycle of length at least $4$ has a chord.

\begin{thm}[{\cite[Theorem 1]{FroMR1171260}}]
A square-free monomial ideal $I_G$ generated by quadrics is $2$-regular if and only if the corresponding graph $G$ is chordal.
\end{thm}

In the literature on positive semidefinite matrix completion, a similar result was proved by Grone-Johnson-S\'a-Wolkowicz, see also Agler-Helton-Rodman-McCullough \cite{AHMRMR960140} and Paulsen-Power-Smith \cite{PPSMR1005860}.
\begin{thm}[{\cite[Theorem 2]{PSDComp}}]
Every nonnegative quadric on $X_G$ is a sum of squares if and only if the graph $G$ is chordal. Equivalently, every partial matrix $\pi_G(a_{ij})$ such that all completely specified symmetric submatrices are {\psd} can be completed to a {\psd} matrix if and only if the graph $G$ is chordal.
\end{thm}

Combining these two theorems with our result on $2$-regularity (Theorem \ref{thm:2regmain}) and the classification of reduced small schemes by Eisenbud-Green-Hulek-Popescu \cite{EGHP}, we get the equivalence of all three statements, which we can phrase in various ways.
\begin{cor}
Let $G$ be a simple graph, let $I_G$ be the Stanley-Reisner ideal of the clique complex of $G$, and let $X_G = \VV(I_G)$. The following statements are equivalent.
\begin{enumerate}[(a)]
 \item The ideal $I_G$ is $2$-regular.
 \item The subspace arrangement $X_G$ is linearly joined.
 \item Every quadric nonnegative on $X_G$ is a sum of squares in $\RR[X_G]$.
 \item Every partial matrix $\pi_G(a_{ij})$ such that all completely specified symmetric submatrices are {\psd} can be completed to a {\psd} matrix.
 \item The graph $G$ is chordal.
\end{enumerate}
\end{cor}

\begin{exm}
For the four cycle $C_4 = ([4],\{ \{1,2\}, \{2,3\},\{3,4\},\{4,1\}\})$ mentioned before, there is a matrix such that all four completely specified $2\times 2$ matrices are {\psd}, but which cannot be completed to a {\psd} $4\times 4$ matrix, because the four cycle is not chordal. An example is 
\[
\begin{pmatrix}
1 & 1 & \ast & -1 \\
1 & 1 & 1 & \ast  \\
\ast & 1 & 1 & 1  \\
-1 & \ast & 1 & 1 
\end{pmatrix}.
\]
A linear functional separating this matrix from the cone of sums of squares on $X_{C_4}$ is given below in Remark \ref{Rmk:extremeray}. If we add one more edge to the graph, we get a chordal graph. So if we specify one more entry, we can complete every matrix such that all symmetric submatrices are {\psd} to a {\psd} matrix.
\end{exm}

The Green-Lazarsfeld index of an arrangement of subspaces $X_G$ can also be described in terms of properties of the graph by a result due to Eisenbud-Green-Hulek-Popescu.

\begin{thm}[{\cite[Theorem 2.1]{EGHPMR2188445}}]
A square-free monomial ideal $I_G$ generated by quadrics satisfies property $N_{2,p}$ if and only if the graph has no induced cycle of length at most $p+2$.
\end{thm}
This theorem implies Fr\"oberg's Theorem on $2$-regularity.
It also allows us to prove the converse to Theorem \ref{Thm:JumpNumberpBp} in the case of monomial ideals. 

\begin{thm}\label{thm:monomialconverse}
Let $G$ be a simple graph and let $I_G$ be the Stanley-Reisner ideal of the clique complex of $G$ and $X_G = \VV(I_G)$ the subspace arrangement. The Hankel index of $X_G$ is the smallest $p\geq 2$ such that $G$ contains an induced cycle of length $p+2$. 
\end{thm}
\begin{proof} By assumption, the graph $G$ does not contain any induced cycle of length $\leq p+1$. By~\cite[Theorem 2.1]{EGHPMR2188445}, we conclude that the Green-Lazarsfeld index $\alpha(X)\geq p-1$ and therefore that $\eta(X)\geq p$ by Theorem~\ref{Thm:JumpNumberpBp}. Let $m=p+2$ and relabel the vertices of $G$ so that a cordless cycle of minimal length in $G$ is given by $([m],\{ \{1,2\},\{2,3\},\dots,\{m,1\}\})$. We restrict our attention to the space $\PP^{m-1} = \VV(x_{m+1},x_{m+2},\dots,x_n)\subset \PP^{n-1}$. The intersection of $X_G$ with this linear subspace is the variety corresponding to the induced cycle by choice of the labels of the vertices. So it is the union of $m$ projective lines
\[
Z = \VV(x_3,x_4,\dots,x_m) \cup \VV(x_4,x_5,\dots,x_m,x_1) \cup \dots \cup \VV(x_1,x_2,\dots,x_{m-2}) \cup \VV(x_2,x_3,\dots,x_{m-1}).
\]
We intersect $Z$ with the hyperplane $H=V(x_1+\dots +x_m)$ obtaining a set of $m$ points $P_1,\dots, P_m$,
\[
Z\cap H = \{[1,-1,0,0,\dots,0],[0,1,-1,0,\dots,0],\dots,[0,\dots,0,1,-1],[-1,0,\dots,0,1]\}.
\]
which span a projective subspace of dimension $m-2=p$. By Theorem~\ref{Thm: deathRay} we conclude that $\eta(X)\leq p$ proving the equality.  
\end{proof}

\begin{remark}\label{Rmk:extremeray} Applying the construction of Theorem~\ref{Thm: deathRay} with the points $P_i$ above one can prove that the following explicit $m\times m$ matrix 
\[
\begin{pmatrix}
\frac{m-2}{m-1} & -1 & 0 & 0 & \dots & 0 & 0 & \frac{1}{m-1} \\
-1 & 2 & -1 & 0 & \dots & 0 & 0 & 0 \\
0 & -1 & 2 & -1 & \dots & 0 & 0 & 0 \\
\vdots & 0 & -1 & 2 & \dots & 0 & 0 & 0\\

0 & 0 & 0 & 0 & \dots & -1 & 2 & -1 \\
\frac{1}{m-1} & 0 & 0 & 0 & \dots & 0 & -1 & \frac{m-2}{m-1}
\end{pmatrix}
\]
is an element of $\Sigma_X^*\setminus P_X^*$. Since its rank equals $\eta(X)$, it must therefore be an extreme ray of $\Sigma_X^*$.
\end{remark}

Now that we can identify the Hankel index of $\Sigma_{X_G}^\ast$, we can also describe the interior of $\Sigma_X$ by rank constraints. In other words, we can describe the set of all quadrics that can be lifted to a positive definite quadric.

\begin{thm}\label{thm:gaussian}
Let $G$ be a simple graph on $n$ vertices and let $m$ be the smallest length of a chordless cycle of $G$. A partial matrix $\pi_G(a_{ij})$ has a positive definite completion if and only if it has the following two properties.
\begin{enumerate}[(a)]
 \item Every completely specified symmetric submatrix is positive definite.
 \item The partial matrix $\pi_G(a_{ij})$ can be completed to a {\psd} matrix of rank greater than $n-m+2$.
\end{enumerate}
\end{thm}

The second condition is void if the graph is chordal. If $G$ is chordal, $m$ is at most $3$, so $n-m+2$ is at least $n-1$.

\begin{proof}
The partial matrix $\pi_G(a_{ij})$ lies in the interior of the image of the {\psd} cone $\Sigma_{X_G}$ if and only if all extreme rays of $\Sigma_{X_G}^\ast$ evaluate to a strictly positive number on $\pi_G(a_{ij})$. The extreme rays of $\Sigma_{X_G}^\ast$ either have rank $1$, in which case they are point evaluations at points in $X_G(\RR)$, or they have rank at least $m-2$ by Theorem \ref{thm:monomialconverse}. The requirement that all point evaluations be positive on the partial matrix corresponds to condition (a) in the claim, because $X_G$ is the union of subspaces corresponding to the cliques in the graph $G$, see Lemma \ref{lem:tautology}. The rank constraint in condition (b) implies that every {\psd} matrix of rank at least $m-2$ must evaluate to a positive number on this completion of rank greater than $n-m+2$, which shows that all other extreme rays are positive on $\pi_G(a_{ij})$, too.
\end{proof}

\begin{rem}
The bound $n-m+2$ on the rank of a completion in part (b) is best possible for any graph $G$. The existence of a cycle of length $m$ guarantees that $\Sigma_{X_G}^\ast$ has an extreme ray of rank $m-2$, see Remark~\ref{Rmk:extremeray}. Such an extreme ray certifies that some matrices with a completion of rank $n-m+2$ get mapped to the boundary of $\Sigma_{X_G}$.
\end{rem}

\begin{rem}
In statistics, the existence of a positive definite completion of a $G$-partial matrix is equivalent to the existence of the maximum likelihood estimator in the Gaussian graphical model, see \cite{Dempster}. Thus, Theorem~\ref{thm:gaussian} gives necessary and sufficient conditions for the existence of the maximum likelihood estimator.
\end{rem}

\subsection{Moment problems on projective varieties.}\label{SEC:Moments}

Let $X$ be a compact topological space and let $C(X,\RR)$ be its algebra of continuous functions with the supremum norm. Fix a filtration $\cF_1\subset \cF_2\subset \dots \subset \cF_k\subset\dots$ of vector subspaces $\mathcal{F}_k\subseteq C(X,\RR)$ such that $\bigcup \mathcal{F}_j\subseteq C(X,\RR)$ is dense. A finite Borel measure $\mu$ on $X$ defines a sequence of linear operators $\ell^{\mu}_k: \mathcal{F}_k\rightarrow \RR$ given by $\ell^{\mu}_k(f):=\int_X fd\mu$. The operator $\ell^{\mu}_k$ is called the moment of order $k$ of $\mu$ and the sequence of operators $(\ell^{\mu}_k)_{k\in \mathbb{N}}$ is called the sequence of moments of $\mu$. The following two basic problems capture the relationship between Borel measures and continuous functions on $X$,

\begin{enumerate}
\item {\it The general moment problem.} Characterize the sequences of operators $(m_k)_{k\in \NN}$ with $m_j:\mathcal{F}_j\rightarrow \RR$ such that there exists a Borel measure $\mu$ on $X$ with $\ell^{\mu}_j=m_j$ for all $j$.
\item {\it The truncated moment problem.} Given an integer $k$ and an operator $m_k: \mathcal{F}_k\rightarrow \RR$ does there exist a Borel measure $\mu$ on $X$ such that $\ell^{\mu}_k = m_k$?
\end{enumerate}

Typically, the set $X$ is a basic closed semialgebraic subset of $\mathbb{R}^n$, the vector spaces $\mathcal{F}_j$ are chosen to be the polynomials of degree at most $j$ in the variables $x_1,\dots, x_n$ and the moments of a measure $\mu$ are specified by giving the values $\ell^{\mu}_j(x^{\alpha})$ on all monomials $x^{\alpha}$ of degree at most $j$, which in the literature are known as moment sequences. The truncated moment problem has many applications, e.g.~in polynomial optimization, probability, finance, and control, see \cite[Chapters~5-14]{Lasserre}.

In this subsection, we study the two problems above when $M$ is a projective scheme and the $\mathcal{F}_j$ are determined by homogeneous polynomials. Our main contribution is a novel sufficient condition for an affirmative answer to the truncated moment problem in this setting.
We begin by clarifying the chosen function spaces in the projective setting. Fix the standard inner product on $\RR^{n+1}$ and let $S\subseteq \RR^{n+1}$ be the unit sphere. Let $X\subseteq \PP^n$ be a scheme over $\RR$. The quotient map $\overline{q}: \wh X(\RR)\rightarrow X(\RR)\subseteq \PP^n(\RR)$ from the affine cone over $X$ to $X$ restricts to a continuous function $q: S\cap \wh X(\RR)\rightarrow X(\RR)$ which identifies antipodal points.
By a polynomial function $f$ on $X(\RR)$ we mean a function of the form $f(y)=F(q^{-1}(y))$ where $F$ is a form {\it of even degree} in the homogeneous coordinate ring $R$ of $X$. We let $R_{\rm even}:=\bigoplus_{k\geq 0} R_{2k}$ and define $\phi: R_{\rm even}\rightarrow C(X(\RR),\RR)$ as the evaluation homomorphism which maps a form $F$ of even degree into its corresponding continuous function on $X(\RR)$. We consider the filtration by the vector spaces $\cF_j = \phi(R_{2j})\subset C(X(\RR),\RR)$. Note that this is, in fact, a filtration because we have $\phi(\|x\|^2 f) = \phi(f)$ for all homogeneous polynomials $f\in R_{2j}$, where $\|x\|^2 = \sum_{i=0}^n x_i^2$. This gives an inclusion $\cF_j \subset \cF_{j+1}$. The union $\bigcup \cF_j$ is dense in $C(X(\RR),\RR)$ by the Stone-Weierstra{\ss} Theorem.

\begin{proposition}
Let $X\subseteq \PP^n$ be a reduced real scheme. 
A sequence of operators $m_k: \mathcal{F}_k\rightarrow \RR$ such that $m_{j+1}$ and $m_j$ agree on $\mathcal{F}_{j}$ for all $j\geq 0$ is the sequence of moments of a Borel measure $\mu$ on $X(\RR)$ if and only if for every $j$ and $f\in \mathcal{F}_j$ which is nonnegative on $X(\RR)$ we have $m_j(f)\geq 0$. Moreover, the sequence $(\ell_k)_{k\in \mathbb{N}}$ determines $\mu$ uniquely. 
\end{proposition}

\begin{proof}
This follows from the Riesz representation theorem: If $F\in R_{2j}$, then $F\|x\|^2\in R_{2(j+1)}$ and $\phi(F\|x\|^2)=\phi(F)$, so $\mathcal{F}_j\subseteq \mathcal{F}_{j+1}$. In particular, $A:=\bigcup \mathcal{F}_j$ is an algebra. It is immediate that this algebra contains the constants and separates points of $X(\RR)$ and therefore, by the Stone-Weierstra{\ss} Theorem, it is a dense subset of $C(X(\RR),\RR)$. Define $m: A\rightarrow \RR$ by $m(g)=m_j(g)$ for $g\in \mathcal{F}_j$. Our assumptions guarantee that $m$ is well defined and that $m$ is nonnegative on nonnegative elements of $A$. Since $X(\RR)$ is compact and the $\mathcal{F}_j$ contain the constants we conclude from~\cite[Theorem 3.1]{Mar} that there is a Borel measure $\mu$ on $X(\RR)$ such that $\ell_k^{\mu}=m_k$ for all $k\geq 0$. The uniqueness is immediate from the density of $A$.
\end{proof}

\begin{theorem}\label{Thm:moment}
Let $X\subseteq \PP^n$ be a reduced real scheme. 
An operator $m_k: \mathcal{F}_k\rightarrow \RR$ satisfies $m_k=\ell^{\mu}_k$ for some measure $\mu$ if and only if $\delta_k:=\ell_k\circ \phi: R_{2k}\rightarrow \RR$ is an element of $P_{X,2k}^*$. If $\delta_k \in \Sigma_{X,2k}^*$, the conclusion holds whenever ${\rm rank}(\delta_k)<\eta(Z)$, where $Z$ is the $k$-th Veronese re-embedding of $X$. 
\end{theorem}

\begin{proof}
If $F\in R_{2k}$ is nonnegative, then $\delta_k(F)=\ell(\phi(F)):=\int_X \phi(F)d\mu$ is obviously nonnegative. Conversely, an element $\delta_k\in P_{X,2k}^*$ is a conic combination of point evaluations at points of $Z(\RR)$, that is, a conic combination of Dirac delta measures at points of $Z(\RR)$ and in a particular comes from a Borel measure as claimed. If $\delta_{k}\in \Sigma_{X,2k}^*$ is not a conic combination of point evaluations, then it must be a conic combination of elements involving some extreme ray which does not have rank one. By definition of the Hankel index, such a ray must have rank at least $\eta(Z)$ and therefore $\delta_k$ must have rank at least $\eta(Z)$, proving the claim. 
\end{proof}

\begin{remark} In order to apply the previous theorem it is useful to have lower bounds for the Hankel index of the $k$-th Veronese re-embedding $Z$ of a given variety $X$. Using Corollary~\ref{Thm:JumpNumber}, such upper bounds can be derived from the Green-Lazarsfeld index of $Z$. This quantity can be obtained by calculating the minimal free resolution of the ideal of $Z$. However, such computations may be infeasible in practice even when the resolution of $X$ is well understood. 
A theorem of Park~\cite{Park} implies that if $X$ is $m$-regular, then $Z$ satisfies $N_{2,2k-m+1}$ for $\frac{m}{2}\leq k\leq m-2$ and $N_{2,k}$ for $k\geq m-1$.  In particular we conclude that the Hankel index of the Veronese re-embeddings of $X$ grow at least linearly, eventually.
\end{remark}

\begin{exm}
For projective toric surfaces, the length of the linear strand of the minimal free resolution is understood combinatorially in terms of the defining lattice polygon $P\subset \RR^2$ by results of Gallego-Purnaprajna, Hering, and Schenck \cite[Theorem 1.3]{GPMR1855753}, \cite{Hering}, \cite[Corollary 2.1]{SchenckMR2084071} . The toric surface $X_P$ has property $N_{2,p}$ if and only if the number of lattice points on the boundary of $P$ is at least $p+3$, (i.e.~$|\partial P\cap \ZZ^2| \geq p + 3$).
Hence, the Green-Lazarsfeld index of $X_P$ is $|\partial P \cap \ZZ^2| - 3$ and we get the lower bound $|\partial P \cap \ZZ^2| - 2$ on the Hankel index of the toric surface $X_P$ by Corollary \ref{Thm:JumpNumber}.
Theorem \ref{Thm:moment} implies that every
moment operator $m_k\colon \cF_k \to \RR$ of rank at most
$|\partial(kP)\cap\ZZ^2|-2$ can be represented by a convex combination
of point measures.
\end{exm}

\begin{exm}
If $P$ is the scaled $2$-simplex $P = \conv\{(0,0),(d,0),(0,d)\}$, then $X_P = \nu_d(\PP^2)$ is the $d$-th Veronese embedding of $\PP^2$. Its Green-Lazarsfeld index is $3d-3$ and the Hankel index is $3d-2$, see \cite{BGorMR3272733}. So every linear functional $\ell\in\Sigma_{X_P}^\ast$ which has rank at most $3d-3$ comes from a measure.
\end{exm}

\begin{exm}
The same bound can be applied to sparse polynomials. For example, take bihomogeneous forms in two sets of two variables of degree $(a,b)$. They correspond to linear forms on the toric surface defined by $P = \conv \{(0,0),(0,b),(a,0),(a,b)\}$. In this setup, Theorem \ref{Thm:moment} states that every linear functional $\ell \in \Sigma_{X_P}^\ast$ of rank at most $|\partial P\cap \ZZ^2| - 3 = 2a + 2b - 3$ comes from a measure.
\end{exm}



\end{document}